\documentclass[a4paper,11pt]{amsart}
\usepackage[left=2.7cm,right=2.7cm,top=3.5cm,bottom=3cm]{geometry}

\usepackage{amsthm,amssymb,amsmath,amsfonts,mathrsfs,amscd,amsbsy,dsfont,verbatim}
\usepackage{stmaryrd}
\usepackage[utf8]{inputenc}
\usepackage[T1]{fontenc}
\usepackage[all,cmtip]{xy}
\usepackage{latexsym}
\usepackage{longtable}
\usepackage[pagebackref]{hyperref}
\usepackage{mathtools}
\usepackage{marginnote}

\usepackage[pagebackref]{hyperref}

\mathtoolsset{showonlyrefs}

\usepackage{graphicx}

\numberwithin{equation}{section}

\newfont{\cyr}{wncyr10 scaled 1100}
\newfont{\cyrr}{wncyr9 scaled 1000}

\theoremstyle{plain}
\newtheorem{theorem}{Theorem}[section]
\newtheorem*{ThmA}{Theorem A}
\newtheorem*{ThmB}{Theorem B}

\newtheorem{proposition}[theorem]{Proposition}
\newtheorem{lemma}[theorem]{Lemma}
\newtheorem{corollary}[theorem]{Corollary}
\newtheorem{conjecture}[theorem]{Conjecture}

\theoremstyle{definition}
\newtheorem{definition}[theorem]{Definition}
\newtheorem{assumption}[theorem]{Assumption}

\theoremstyle{remark}
\newtheorem{remark}[theorem]{Remark}
\newtheorem{notation}[theorem]{Notation}

\newtheorem{remark/notation}[theorem]{Remark/Notation}
\newtheorem{notation/convention}[theorem]{Notation/Convention}

\newcommand{\Q}{\mathds Q}
\newcommand{\N}{\mathds N}
\newcommand{\Z}{\mathds Z}
\newcommand{\R}{\mathds R}
\newcommand{\C}{\mathds C}

\newcommand{\T}{\mathds T}

\newcommand{\defeq}{\vcentcolon=}

\DeclareMathOperator{\Pic}{Pic}
\DeclareMathOperator{\End}{End}
\DeclareMathOperator{\Aut}{Aut}
\DeclareMathOperator{\Frob}{Frob}

\DeclareMathOperator{\Hom}{Hom}

\DeclareMathOperator{\Gal}{Gal}
\DeclareMathOperator{\GL}{GL}

\DeclareMathOperator{\Sel}{Sel}

\DeclareMathOperator{\M}{M}

\DeclareMathOperator{\BKK}{BK}
\DeclareMathOperator{\Nek}{Nek}

\DeclareMathOperator{\CH}{CH}
\DeclareMathOperator{\AJ}{AJ}

\DeclareMathOperator{\Tam}{Tam}

\DeclareMathOperator{\Heeg}{Heeg}

\DeclareMathOperator{\im}{im}

\DeclareMathOperator{\corank}{corank}

\DeclareMathOperator{\length}{length}

\DeclareMathOperator{\Tors}{Tors}

\DeclareMathOperator{\JL}{\mathrm{JL}}

\newcommand{\val}{\mathrm{val}}
\newcommand{\res}{\mathrm{res}}

\newcommand{\cyc}{{\rom{cyc}}}

\newcommand{\ord}{\mathrm{ord}}
\newcommand{\fin}{\mathrm{fin}}

\newcommand{\mot}{\mathrm{mot}}

\newcommand{\an}{\mathrm{an}}
\newcommand{\alg}{\mathrm{alg}}

\newcommand{\reg}{\mathtt{reg}}
\newcommand{\loc}{\mathrm{loc}}

\newcommand{\sing}{\mathrm{sing}}

\newcommand{\Sha}{\mbox{\cyr{X}}}

\usepackage[usenames]{color}
\definecolor{Indigo}{rgb}{0.2,0.1,0.7}
\definecolor{Violet}{rgb}{0.5,0.1,0.7}
\definecolor{White}{rgb}{1,1,1}
\definecolor{Green}{rgb}{0.1,0.9,0.2}

\newcommand{\longmono}{\mbox{\;$\lhook\joinrel\longrightarrow$\;}}

\newcommand{\longepi}{\mbox{\;$\relbar\joinrel\twoheadrightarrow$\;}}

\newcommand{\smallmat}[4]{\bigl(\begin{smallmatrix}#1&#2\\#3&#4\end{smallmatrix}\bigr)}

\newcommand{\cB}{\mathcal B}

\newfont{\gotip}{eufb10 at 12pt}

\newcommand{\cO}{{\mathcal O}}
\newcommand{\cM}{{\mathscr M}}

\newcommand{\p}{\mathfrak{p}}

\newcommand{\Sym}{\operatorname{Sym}}

\DeclareMathOperator{\Kol}{Kol}
\DeclareMathOperator{\GS}{GS}

\DeclareMathOperator{\Reg}{Reg}

\newcommand{\rom}{\mathrm}

\newcommand{\et}{\text{\'et}}
\newcommand{\MM}{\mathcal{M}}

\newcommand{\RR}{\mathscr{R}}

\makeatletter
\@namedef{subjclassname@2020}{%
  \textup{2020} Mathematics Subject Classification}
\makeatother

\include{thebibliography}

\begin{document}

\title[Kolyvagin's conjecture for modular forms]{Kolyvagin's conjecture for modular forms}
\author{Matteo Longo, Maria Rosaria Pati and Stefano Vigni}

\thanks{The authors are partially supported by PRIN 2022 ``The arithmetic of motives and $L$-functions'' and by the GNSAGA group of INdAM. The research by the second and third authors is partially supported by the MUR Excellence Department Project awarded to Dipartimento di Matematica, Universit\`a di Genova, CUP D33C23001110001.}

\begin{abstract}
Our main result in this article is a proof (under mild technical assumptions) of an analogue for $p$-adic Galois representations attached to a newform $f$ of even weight $k\geq4$ of Kolyvagin's conjecture on the $p$-indivisibility of derived Heegner points on elliptic curves, where $p$ is a prime number that is ordinary for $f$. Our strategy, which is inspired by work of W. Zhang in weight $2$, is based on a variant for modular forms of the congruence method originally introduced by Bertolini--Darmon to prove one divisibility in the anticyclotomic Iwasawa main conjecture for rational elliptic curves. We adapt to higher (even) weight modular forms this approach via congruences, building crucially on results of Wang on the indivisibility of Heegner cycles over Shimura curves. Then we offer an application of our results on Kolyvagin's conjecture to the Tamagawa number conjecture for the motive of $f$ and describe other (standard) consequences on structure theorems for Bloch--Kato--Selmer groups, $p$-parity results and converse theorems for $f$. Since in the present paper we need $p>k+1$, our main theorem and its applications can be viewed as complementary to results obtained by the first and third authors in their article on the Tamagawa number conjecture for modular motives, where Kolyvagin's conjecture was proved (in a completely different way exploiting the arithmetic of Hida families) under the assumption that $k$ is congruent to $2$ modulo $2(p-1)$, which forces $p<k$.  In forthcoming work, we will use results contained in this paper to prove (under analogous assumptions) the counterpart for an even weight newform $f$ of Perrin-Riou's Heegner point main conjecture for elliptic curves (``Heegner cycle main conjecture'' for $f$).
\end{abstract}

\address{Dipartimento di Matematica, Universit\`a di Padova, Via Trieste 63, 35121 Padova, Italy}
\email{mlongo@math.unipd.it}
\address{Dipartimento di Matematica, Universit\`a di Genova, Via Dodecaneso 35, 16146 Genova, Italy}
\email{mariarosaria.pati@unige.it}
\address{Dipartimento di Matematica, Universit\`a di Genova, Via Dodecaneso 35, 16146 Genova, Italy}
\email{stefano.vigni@unige.it}



\maketitle



\section{Introduction} \label{Intro}

In \cite{kolyvagin-selmer}, Kolyvagin proposed a conjecture predicting the $p$-indivisibility of his system of derived Galois cohomology classes built out of Heegner points on rational elliptic curves, where $p$ is a prime number. This conjecture, which has several remarkable consequences on the arithmetic of elliptic curves, was proved (under standard assumptions) by W. Zhang for good reduction primes (\cite{zhang-selmer}) and by Skinner--Zhang for multiplicative reduction primes (\cite{SZ}). The strategy in these works is based on the congruence method originally introduced by Bertolini--Darmon to prove one divisibility in the anticyclotomic Iwasawa main conjecture for rational elliptic curves (\cite{BD-IMC}). As an application of their work on Kolyvagin's conjecture, Zhang and Skinner--Zhang proved the $p$-part of the Birch and Swinnerton-Dyer formula for elliptic curves in analytic rank $1$; under different arithmetic assumptions and in various degrees of generality, this formula was proved independently also by Berti--Bertolini--Venerucci (\cite{BBV}), Jetchev--Skinner--Wan (\cite{JSW}) and Castella (\cite{castella-BSD}). 

Our main result in this paper is a proof (under mild technical conditions) of an analogue for the $p$-adic Galois representations attached to an even weight newform $f$ of Kolyvagin's conjecture, where $p$ is a prime number that is ordinary for $f$. Such an analogue was first proposed (in a slightly less general context) by Masoero in \cite{Masoero}, where structure theorems for Selmer and Shafarevich--Tate groups of modular forms are proved. As will be apparent, our strategy is inspired by the work of Zhang and of Berti--Bertolini--Venerucci in weight $2$ and builds crucially on results of Wang on the $p$-indivisibility of Heegner cycles over Shimura curves (\cite{wang}). It is worth remarking that Kolyvagin's conjecture for even weight newforms was studied earlier in the article by the first and third authors on the Tamagawa number conjecture for modular motives (\cite{LV-TNC}). In \cite{LV-TNC}, a proof of various cases of this conjecture was the key technical step towards proving (conditionally on very specific instances of two general conjectures in arithmetic algebraic geometry, namely, bijectivity of $p$-adic regulator maps and injectivity of $p$-adic Abel--Jacobi maps) the $p$-part of the Tamagawa number conjecture for the (homological) motive attached to a newform of even weight $k\geq4$ in analytic rank $1$. In particular, since in the present paper we need $p>k+1$, our main theorem here can be seen as complementary to the results in \cite{LV-TNC}, where Kolyvagin's conjecture was proved (in a completely different way exploiting the arithmetic of $p$-adic Hida families of modular forms, \emph{cf.} Remark \ref{strategy-rem}) under the assumption that $k$ is congruent to $2$ modulo $2(p-1)$, which is a condition on the $p$-adic weight space of Coleman--Mazur (\cite{ColMaz}) and forces the inequality $p<k$.

The organisation of this paper is as follows. Section \ref{reciprocity-sec} is mainly devoted to background material on Shimura curves, classical and quaternionic modular forms, Selmer groups and Heegner cycles; our exposition here follows \cite{BBV} quite closely. We also review, under mild conditions on the prime $p$ and residual $p$-adic representations (Assumption \ref{ass}) that are in force throughout the article and are higher weight counterparts of analogous assumptions appearing in \cite{zhang-selmer}, the two reciprocity laws \emph{\`a la} Bertolini--Darmon (\cite{BD-IMC}) as extended to our higher weight setting by Wang (\cite{wang}). 

Section \ref{kolyvagin-sec} is the heart of the paper. We define our Kolyvagin systems $\kappa_S$ of derived Galois cohomology classes built out of quaternionic Heegner cycles (here $S$ is allowed to vary over a certain set $\mathcal{P}^\mathrm{indef}$ of square-free products of ``admissible'' primes), we formulate an analogue of Kolyvagin's conjecture for the newform $f$, both in ``classical'' form (Conjecture \ref{kolyvagin-conj}) and in strong form (Conjecture \ref{strong-kolyvagin-conj}), then we prove (under Assumption \ref{ass}) the strong version of the conjecture. More precisely, under the above-mentioned assumption, our main result is the following. 

\begin{ThmA}
For all $S\in\mathcal{P}^\mathrm{indef}$, the strong form of Kolyvagin's conjecture for $\kappa_S$ holds true. In particular, $\kappa_S\not=\{0\}$.
\end{ThmA}

This is Theorem \ref{kol-thm} in the main body of the text. As we pointed out before, a key ingredient in our proof of Theorem A is a result of Wang on the $p$-indivisibility of Heegner cycles on Kuga--Sato varieties fibered over Shimura curves (\cite{wang}), which essentially says that (the strong form of) Kolyvagin's conjecture for $\kappa_S$ holds true if the rank of a certain Selmer group is $1$. This result provides the first step in an inductive argument that is closely inspired by the one adopted by Zhang in the weight $2$ case (\cite{zhang-selmer}); this argument is based on a process that (following Zhang) we call ``triangulation'' of Selmer groups (\S \ref{triangulation-subsec}). We remark that the prime number $p$ is assumed to be ordinary for $f$, as this property is imposed in \cite{wang}, whose results play a central role here; an extension of Wang's results in Selmer rank $1$ (and possibly of our work on Kolyvagin's conjecture as well) to the case of non-ordinary primes is expected to appear as part of Enrico Da Ronche's PhD thesis at Universit\`a di Genova.

The last two sections of this article describe arithmetic consequences of the validity of Kolyvagin's conjecture for $f$. Section \ref{tamagawa-sec} offers an application of Theorem A to the $p$-part of the Tamagawa number conjecture of Bloch--Kato (\cite{BK}) and Fontaine--Perrin-Riou (\cite{FPR}) for the motive $\MM$ of $f$. More precisely, let $r_\alg(\MM)$ and $r_\an(\MM)$ be the algebraic rank and the analytic rank of $\MM$, respectively (see \S \ref{TNC-statement-subsec} for definitions); our contribution in the direction of the Tamagawa number conjecture for $\MM$ can be stated (in a somewhat simplified fashion) as follows.

\begin{ThmB}
Suppose that the assumptions in \S \ref{TNC-ass} are satisfied. If $r_\an(\MM)=1$, then $r_\alg(\MM)=1$ and the $p$-part of the Tamagawa number conjecture for $\MM$ is true. 
\end{ThmB}

This result corresponds to Theorem \ref{p-TNC}. With Theorem A available, our proof of Theorem B works exactly as in \cite{LV-TNC}, where details can be found (see \S \ref{TNC-remarks} for some remarks on our overall strategy).

Finally, Section \ref{consequences-sec} gives further (standard) applications of Theorem A to the arithmetic of $f$: a structure theorem for Bloch--Kato--Selmer groups, $p$-parity results, $p$-converse theorems. As the proofs are identical to those of analogous results in \cite{LV-TNC}, in this paper we just focus on the statements and refer again to \cite{LV-TNC} for details.

We conclude this introduction by remarking that results obtained in this paper will play a major role in our forthcoming work \cite{LPV}, where we prove (under assumptions similar to those required here) the counterpart for an even weight newform $f$ of Perrin-Riou's Heegner point main conjecture for elliptic curves (``Heegner cycle main conjecture'' for $f$).

\subsection{Notation and conventions} \label{notation-subsec}

Throughout this paper, we fix an algebraic closure $\bar\Q$ of $\Q$, an embedding $\iota_\infty:\bar\Q\hookrightarrow\C$ and set $G_\Q\defeq\Gal(\bar\Q/\Q)$ for the absolute Galois group of $\Q$. For any number field $L$, we choose an embedding $\iota_L:L\hookrightarrow\bar\Q$, which will allow us to view $L$ as a subfield of $\bar\Q$. We fix also a prime number $p$, an algebraic closure $\bar\Q_p$ of $\Q_p$ and an embedding $\iota_p:\bar\Q\hookrightarrow\bar\Q_p$. Moreover, the symbol $\cO$ will stand for the ring of integers of a number field and $\wp$ for a maximal ideal of $\cO$ of residue characteristic $p$; as customary, $\cO_\wp$ will denote the completion of $\cO$ at $\wp$. Finally, the term ``ring'' will always mean ``commutative ring with unity'', unless otherwise stated.

\subsection*{Acknowledgements} 

We would like to thank Enrico Da Ronche for his interest in our work and for enlightening conversations on some of the topics of this paper.


\section{Heegner cycles and explicit reciprocity laws} \label{reciprocity-sec}

We collect in this section all background material that we will need concerning the explicit reciprocity laws proved in \cite{wang}, extending to higher weight modular forms results for elliptic curves (\emph{i.e.}, for weight $2$ newforms with rational Fourier coefficients) from the seminal paper \cite{BD-IMC} by Bertolini--Darmon that were later generalised, \emph{e.g.}, in \cite{ChHs2}, \cite{Longo-AIF}, \cite{Longo-Padova}, \cite{Longo-CMH}, \cite{Nek-LR}, \cite{zhang-selmer}. We closely follow the spirit of \cite{BBV} and \cite{BLV} for the exposition of this material. We fix a pair of coprime square-free integers $D$ and $M$ such that $p\nmid MD$. Notation from \S \ref{notation-subsec} is in force.

\subsection{Shimura curves and sets} \label{sec2.1}

Let $\cB=\cB_D$ be the (unique, up to isomorphism) quaternion algebra over $\Q$ of discriminant $D$. If the number of prime factors of $D$ is \emph{odd} (respectively, \emph{even}), then $\cB$ is definite (respectively, indefinite) and we say that we are in the \emph{definite} (respectively, \emph{indefinite}) case. Let us fix an Eichler order $\mathcal{R}=\mathcal{R}_{M,D}$ of level $M$ in $\cB$ and a maximal order $\cO_\cB$ of $\cB$ containing $\mathcal{R}$.
For each prime $\ell\nmid D$, choose an isomorphism of $\Q_\ell$-algebras
\begin{equation}\label{iq}
 i_\ell:\cB_\ell\defeq\cB\otimes_\Q\Q_\ell\overset\simeq\longrightarrow\M_2(\Q_\ell) 
\end{equation}  
such that the image of $\mathcal{R}_\ell\defeq\mathcal{R}\otimes_\Z\Z_\ell$ via $i_\ell$ is equal to the subgroup of $\M_2(\Z_\ell)$ of matrices that are upper triangular modulo $\ell^{\,\ord_\ell(M)}$. In the indefinite case we also fix an isomorphism of $\R$-algebras
\begin{equation} \label{iinfty}
i_\infty:\cB_\infty\defeq\cB\otimes_\Q\R\overset\simeq\longrightarrow\M_2(\R). 
\end{equation}
Let $K$ be an imaginary quadratic field. Let us fix an embedding of $\Q$-algebras $\psi_K:K\hookrightarrow \cB$ that allows us to identify $K$ with a subalgebra of $\cB$: we write $K\subset\cB$ instead of $\psi_K(K)\subset\cB$. Let $t\mapsto \bar t$ be the main involution of $\cB$, which coincides with the action of $\Gal(K/\Q)$ on $K$. Now choose $J_\cB\in \cB^\times$ satisfying the following conditions (recall that the prime $p$ was fixed at the beginning of the section):  
\begin{itemize}
\item $\cB=K\oplus K\cdot J_\cB$ as $\Q$-vector spaces; 
\item $J_\cB^2=\beta$ with $\beta\in\Q^\times$ and $\beta<0$; 
\item $J_\cB t=\bar{t}J_\cB$ for all $t\in K$; 
\item $\beta\in (\Z_\ell^\times)^2$ for all $\ell\,|\,Mp$ and $\beta\in \Z_\ell^\times$ for all $\ell\,|\,D$.
\end{itemize}
Let $-D_K<0$ denote the discriminant of $K$ and set $\delta_K\defeq\sqrt{-D_K}$; put $D'_K\defeq D_K$ if $D_K$ is odd and $D'_K\defeq D_K/2$ if $D_K$ is even. Define $\theta\defeq\frac{D_K'+\delta_K}{2}$. We require the isomorphisms $i_\ell$ in \eqref{iq} to further satisfy the conditions $i_\ell(\theta)=\smallmat{\theta+\bar\theta}{-\theta\bar\theta}{1}{0}$ and $i_\ell(J_\cB)=\sqrt{\beta}\cdot\smallmat{-1}{\theta+\bar\theta}{0}{1}$ for all primes $\ell\,|\,Mp$ (note that $\sqrt{\beta}\in \Z_\ell^\times$ for all such $\ell$). 

Let $\widehat{\Z}\defeq\varprojlim_N\Z/N\Z$ be the profinite completion of $\Z$ and for any abelian group $A$ set $\widehat{A}\defeq A\otimes_\Z\widehat{\Z}$. The group $\widehat{\mathcal{R}}^\times$ acts by left multiplication on $\widehat{\cB}^\times$ and trivially on $\Hom(\C,\cB_\infty)$, while $\cB^\times$ acts by right multiplication on $\widehat{\cB}^\times$ and by conjugation on $\Hom(\C,\cB_\infty)$. Define the double coset space 
\[ Y_0(M,D)(\C)\defeq\widehat{\mathcal{R}}^\times\big\backslash\bigl(\widehat{\cB}^\times\times\Hom(\C,\cB_\infty)\bigr)\big/\cB^\times. \]
As we shall recall below, the geometric properties of $Y_0(M,d)$ largely depend on whether the quaternion algebra $\cB$ is definite or indefinite.

\subsubsection{The indefinite case}\label{sec2.1.1}

Let $\mathcal{H}\defeq\bigl\{z\in\C\mid\Im(z)>0\bigr\}$ be the complex upper half plane. In the indefinite case, there are canonical bijections
\begin{equation} \label{bijections-eq}
Y_0(M,D)(\C)\simeq\widehat{\mathcal{R}}^\times\big\backslash\bigl(\widehat{\cB}^\times\times\mathcal H\bigr)\big/\cB^\times\simeq\Gamma_0(M,D)\backslash\mathcal{H}, 
\end{equation}
where $\Gamma_0(M,D)$ is the subgroup of $i_\infty\bigl(\widehat{\mathcal{R}}^\times\cap\cB^\times\bigr)\subset\GL_2(\R)$ consisting of all elements of determinant $1$ and the second bijection comes from the strong approximation theorem. The set $Y_0(M,D)(\C)$ inherits a Riemann surface structure and admits a model defined over $\Q$, which we denote simply by $Y_0(M,D)$. We call $Y_0(M,D)$ the \emph{Shimura curve} of level $M$ and discriminant $D$ attached to $(\cB,\mathcal{R})$. If $D=1$, then this is the classical (open) modular curve $Y_0(M)=Y_0(M,1)$ and we let $X_0(M)$ denote its canonical (Baily--Borel) compactification that is obtained by adding finitely many cusps.  

Suppose $D>1$. Then $Y_0(M,D)$ is a projective curve admitting a model $\mathcal{X}_0(M,D)$ over $\Z[1/MD]$ that coarsely represents the moduli problem attaching to each $\Z[1/MD]$-scheme $S$ the set of isomorphism classes of triples $(A,\iota_A,C)$ where
\begin{itemize}
\item $A$ is a polarized abelian surface over $S$,
\item $\iota_A:\cO_\cB\hookrightarrow \End_S(A)$ is an embedding,
\item $C$ is an $\cO_\cB$-stable locally cyclic subgroup of $A[M]$ of order $M^2$. 
\end{itemize} 
Here, for any integer $m\geq1$, $A[m]$ denotes the $m$-torsion subgroup scheme of $A$; the map $\iota_A$ is called a \emph{quaternionic action}. Let now $d\geq 5$ be an auxiliary integer and write $\mathcal{X}_d(M,D)$ for the projective $\Z[1/MDd]$-scheme representing the moduli problem that associates with each $\Z[1/MDd]$-scheme $S$ the set of isomorphism classes of $4$-tuples $(A,\iota_A,C,\nu_d)$ where the triple $(A,\iota_A,C)$ is as before and 
\begin{itemize}
\item $\nu_d:(\cO_\cB/d\cO_\cB{)}_S\rightarrow A[d]$ is an isomorphism of $\mathcal{O}_\cB$-stable group schemes.
\end{itemize}
Here $(\cO_\cB/d\cO_\cB{)}_S$ is the constant $\cO_\cB/d\cO_\cB$-valued group scheme over $S$. The isomorphism $\nu_d$ is called a \emph{full level $d$-structure}. There is a covering map $\mathcal{X}_d(M,D)\rightarrow\mathcal{X}_0(M,D)$ that is defined over $\Z[1/MDd]$ and corresponds (in the moduli interpretation) to the canonical map forgetting the $\nu_d$-structure. Denote by $\mathcal{A}_d(M,D)\rightarrow\mathcal{X}_d(M,D)$ the universal object, so that $\mathcal{A}_d(M,D)$ is an abelian surface over $\Z[1/MDd]$ equipped with a quaternionic action, a locally cyclic subgroup scheme of order $M^2$ and a full level $d$-structure. For any integer $k\geq 2$, let $\mathcal{W}_{k,d}(M,D)\rightarrow \mathcal{X}_d(M,D)$ be the Kuga--Sato variety of weight $k$ over $\mathcal{X}_d(M,D)$, which is the $(k-2)/2$-fold self-product of 
$\mathcal{A}_d(M,D)$ over $\mathcal{X}_d(M,D)$. Let $W_{k,d}(M,D)\rightarrow X_d(M,D)$ denote the generic fiber of $\mathcal{W}_{k,d}(M,D)$; we also write $A_d(M,D)\rightarrow X_d(M,D)$ for the generic fiber of $\mathcal{A}_{d}(M,D)\rightarrow\mathcal{X}_d(M,D)$. The dimension of $W_{k,d}(M,D)$ is $k-1$. 

Suppose $D=1$. Then $X_0(M)$ is a projective curve having a model $\mathcal{X}_0(M)$ over $\Z[1/M]$ that is obtained as the canonical compactification of a model $\mathcal{Y}_0(M)$ over $\Z[1/M]$ of the open modular curve $Y_0(M)$; the curve $\mathcal Y_0(M)$ coarsely represents over $\Z[1/M]$ the moduli problem that associates with a $\Z[1/M]$-scheme $S$ the isomorphism classes of pairs $(E,C)$ of the following kind:
\begin{itemize}
\item $E$ is an elliptic curve over $S$;
\item $C$ is a locally cyclic subgroup of $E[M]$ of order $M$. 
\end{itemize}
Let as before $d\geq 5$ be an auxiliary integer and denote $\mathcal{Y}_d(M)=\mathcal{Y}_d(M,1)$ the projective $\Z[1/Md]$-scheme representing the moduli problem that associates with a $\Z[1/Md]$-scheme $S$ the isomorphism classes of triples $(E,C,\nu_d)$, where the pair $(E,C)$ is as before and 
\begin{itemize}
\item $\nu_d:(\Z/d\Z)^2\overset\simeq\longrightarrow E[d]$ is an isomorphism of group schemes.
\end{itemize}
The isomorphism $\nu_d$ is again called a \emph{full level $d$-structure}; as above, there is a covering map $\mathcal{Y}_d(M)\rightarrow\mathcal{Y}_0(M)$ defined over $\Z[1/Md]$ that corresponds (in the moduli interpretation) to the map forgetting the $\nu_d$-structure. Let $\mathcal{E}_d(M)=\mathcal{E}_d(M,1)\rightarrow\mathcal{Y}_d(Y)$ be the universal object. Denote $\mathcal{X}_d(M)=\mathcal{X}_d(M,1)$ the canonical compactification of $\mathcal{Y}_d(M)$ and let $\bar{\mathcal{E}}_d(M)\rightarrow\mathcal{X}_d(M)$ denote the generalised universal elliptic curve equipped with a full $d$-level structure and a locally cyclic subgroup of order $M$. Let $\mathcal{W}_{k,d}(M)=\mathcal{W}_{k,d}(M,1)\rightarrow \mathcal{X}_d(M)$ be the Kuga--Sato variety of weight $k\geq2$ over $\mathcal{X}_d(M)$ defined as the desingularization of the $(k-2)$-fold self-product of $\bar{\mathcal{E}}_d(M)$ over $\mathcal{X}_d(M)$. As before, write $W_{k,d}(M)=W_{k,d}(M,1)\rightarrow X_d(M)$ for the generic fiber of $\mathcal{W}_{k,d}(M)$ and $E_d(M)\rightarrow X_d(M)$ for the generic fiber of $\mathcal{E}_d(M)\rightarrow \mathcal{X}_d(M)$. The dimension of $W_{k,d}(M)$ is again $k-1$. 
 
\subsubsection{The definite case} 

In the definite case, write $\mathcal{R}_1,\dots,\mathcal{R}_h$ for the conjugacy classes of oriented Eichler orders of level $M$ in $\cB$; moreover, for $j=1,\dots,h$ set $\Gamma_j\defeq\mathcal{R}_j^\times/\{\pm 1\}$. Each $\Gamma_j$ is a finite group and there is an isomorphism 
\[ Y_0(M,D)(\C)\simeq\coprod_{j=1}^h\Gamma_j\backslash\Hom(\C,\cB_\infty). \]
We attach to $\cB$ a conic $\mathcal{C}_{/\Q}$ as follows. For every $\Q$-algebra $A$, we let $\mathcal{C}(A)$ be the subset of $\cB\otimes_\Q A$ consisting of all non-zero elements having trace and norm $0$ modulo $A^\times$. Define $\mathcal{C}_j\defeq\Gamma_j\backslash\mathcal{C}$ for $j=1,\dots,h$ and 
\[ Y_0(M,D)\defeq\coprod_{j=1}^h\mathcal{C}_j, \]
which is a curve defined over $\Q$. Then, as the notation suggests, $Y_0(M,D)(\C)$ is the set of complex points of $Y_0(M,D)$. 
We call $Y_0(M,D)$ the \emph{Gross curve} of level $M$ and discriminant $D$ attached to $(\cB,\mathcal{R})$. See, \emph{e.g.}, \cite[\S 2.1]{BLV} for details.

\subsection{Modular forms on quaternion algebras} \label{modular-quaternion-subsec}

We briefly review the basics of the theory of modular forms on quaternion algebras over $\Q$.

\subsubsection{Pairings on polynomial rings} \label{pairings-subsubsec}

Let $k\geq 2$ be an even integer and let $A$ be a ring. Set $r\defeq k-2$ and $L_r(A)\defeq\Sym^r(A^2)$; the $A$-module $L_r(A)$ can be identified with the $A$-module $A[X,Y{]}_r$ of homogeneous polynomials of degree $r$ in two variables $X,Y$ over $A$. Let $\mathbf{v}_j\defeq X^{r/2-j}Y^{r/2+j}$ for $j=0,\dots,r/2$ be the standard basis elements of $L_r(A)$. Denote by 
\begin{equation} \label{rho-eq}
\rho_k:\GL_2(A)\longrightarrow\Aut_A\bigl(L_r(A)\bigr) 
\end{equation}
the representation defined by $\rho_k(g)(P)\defeq\det^{-(r/2)}(g)(P|g)$, where if $P$ is a polynomial in $A[X,Y{]}_r$ and $(X,Y)g$ is left matrix multiplication, then $(P|g)(X,Y)\defeq P\bigl((X,Y)g\bigr)$. 

As in \S \ref{notation-subsec}, let $p$ be a prime number. Suppose that $p>k+1$. If $A$ is a $\Z_{(p)}$-algebra, then we define a perfect pairing ${\langle\cdot,\cdot\rangle}_k:L_r(A)\times L_r(A)\rightarrow A$ by setting
\[ \bigg\langle \sum_{i=1}^ra_i\mathbf{v}_i,\sum_{j=1}^rb_j\mathbf{v}_j\bigg\rangle_{\!k}\defeq\sum_{-\frac{k-2}{2}\leq n\leq \frac{k-2}{2}}a_nb_{-n}\frac{\Gamma(k/2+n)\Gamma(k/2-n)}{\Gamma(k-1)}. \]
This pairing satisfies the rule $\big\langle\rho_k(g)P_1,\rho_k(g)P_2\big\rangle_k={\langle P_1,P_2\rangle}_k$ for all $P_1,P_2\in L_r(A)$ and all $g\in\GL_2(A)$. 

\subsubsection{The indefinite case} 

For a ring $A$, let $\mathcal{L}_r(A)$ denote the locally constant sheaf on $X_0(M,D)$ (respectively, $Y_0(M)$) when $D>1$ (respectively, $D=1$) associated with the $\GL_2(A)$-module $L_r(A)$. Write $H^1_\mathrm{par}\bigl(X_0(M,D),\mathcal L_r(A)\bigr)$ for the parabolic cohomology group of $X_0(M,D)$ with coefficients in $\mathcal L_r(A)$, which coincides with $H^1\bigl(X_0(M,D),\mathcal L_r(A)\bigr)$ when $D>1$ and with the natural image in $H^1\bigl(Y_0(M),\mathcal{L}_r(A)\bigr)$ of the compactly supported cohomology group $H^1_\mathrm{cpt}\bigl(Y_0(M),\mathcal{L}_r(A)\bigr)$ when $D=1$. 

\begin{definition}
The $A$-module of \emph{$A$-valued modular forms} of weight $k$ and level $M$ on $\cB$ is $S_k(M,D;A)\defeq H^1_\mathrm{par}\bigl(X_0(M,D),\mathcal{L}_r(A)\bigr)$.
\end{definition}

We just set $S_k(M,D)\defeq S_k(M,D;\C)$ for the $\C$-vector space of modular forms of weight $k$ and level $M$ on $\cB$.  
 
\subsubsection{The definite case} \label{definite-subsubsec}

Suppose that $\cB$ is definite. We introduce the $A$-module of modular forms of weight $k$ and level $M$ on $\cB$, and then make a comparison with other definitions in \cite{ChHs1} (\emph{cf.} also \cite{ChHs2}, \cite{wang}). As above, $A$ is a ring.

\begin{definition} \label{defMF}
The space of \emph{quaternionic modular forms} of weight $k$ and level $M$ on $\cB$ with coefficients in $A$ is the $A$-module $S_k(M,D;A)$ of all functions $\cB^\times\backslash\widehat\cB^\times/\widehat{\mathcal{R}}^\times\rightarrow L_k(A)$. 
\end{definition}

Now we compare Definition \ref{defMF} with the corresponding definition in \cite{ChHs1} when $A=\C$. Recall the element $J_\cB$ such that $\cB=K\oplus KJ_\cB$ and define the embedding $\iota_{\cB,K}:\cB\hookrightarrow \M_2(K)$ by $\iota_{\cB,K}(a+b J_\cB)\defeq\smallmat {a}{b\beta}{\bar b}{\bar a}$. We thus obtain a morphism $\cB_\infty^\times\rightarrow\GL_2(\C)$ by composing each component of the image of $\iota_{\cB,K}$ with the fixed morphism $\iota_\infty:\bar\Q\hookrightarrow\C$; by further composing with the representation $\rho_k$ from \eqref{rho-eq}, we get a representation  $\rho_{k,\infty}:\cB_\infty^\times\rightarrow\End_\C\bigl(\mathcal{L}_r(\C)\bigr)$ given by $\rho_{k,\infty}=\rho_k\circ\iota_\infty\circ\iota_{\cB,K}$. In \cite{ChHs1}, complex-valued modular forms are defined as functions 
$\Phi:\widehat\cB^\times\rightarrow L_r(\C)$ such that $\Phi(agu)=\rho_{k,\infty}(a)\cdot\Phi(g)$ for all $g\in\widehat\cB^\times$, $a\in\cB^\times$ and $u\in \widehat{\mathcal{R}}^\times$; 
more generally, if $A$ is a $K$-algebra, then $\rho_{k,\infty}(a)$ belongs to $\End_K\bigl(\mathcal{L}_r(K)\bigr)$ for $a\in \cB^\times$, so we may define $A$-valued modular forms as functions
\begin{equation} \label{CVMF}
\Phi:\widehat{\mathcal B}^\times\longrightarrow L_r(A)
\end{equation} 
such that $\Phi(agu)=\rho_{k,\infty}(a)\Phi(g)$ for all $g\in\widehat\cB^\times$, $a\in\cB^\times$ and $u\in\widehat{\mathcal{R}}^\times$. 

In \cite{ChHs1}, $p$-adic modular forms are introduced as follows. Let $\p$ be the prime of $K$ above $p$ corresponding to the embedding $\iota_p\circ\iota_K$, notation being as in \S \ref{notation-subsec}. With a slight notational abuse, for $u\in\widehat\cB^\times$ we let $u_p\in \GL_2(K_\p)$ denote the $\p$-component of $i_p(u)$, where $K_\p$ is the completion of $K$ at $\p$; therefore, if $u\in \widehat{\mathcal{R}}^\times$ and $A$ is an $\cO_{K_\p}$-algebra, where $
\cO_{K_\p}$ is the valuation ring of $K_\p$, then $\rho_k(u_p)$ belongs to $\End_{A}\bigl(\mathcal{L}_r(A)\bigr)$. For an $\cO_{K_\p}$-algebra $A$, a \emph{$p$-adic quaternionic modular form} of weight $k$ and level $M$ on $\cB$ with coefficients in $A$ is a function 
\begin{equation} \label{padicMF}
\Psi:\widehat\cB^\times\longrightarrow L_r(A)
\end{equation} 
such that $\Psi(agu)=\rho_k(u_p^{-1})\cdot\Psi(g)$ for all $g\in\widehat\cB^\times$, $a\in B^\times$ and $u\in \widehat{\mathcal{R}}^\times$. Composing the embedding $\iota_{\cB,K}:\cB\hookrightarrow\M_2(K)$ with the fixed embedding $\iota_p:\bar\Q\hookrightarrow\bar\Q_p$ (as before, on each component of the image), we obtain an embedding of $\Q_p$-algebras $\cB^\times\hookrightarrow \GL_2(\bar\Q_p)$. Further composing with $\rho_k$, we obtain a representation 
\[ \rho_{k,p}:\cB^\times\longrightarrow \End_{\bar\Q_p}\bigl(\bar\Q_p\bigr) \] 
defined by $\rho_{k,p}\defeq\rho_k\circ \iota_p\circ\iota_{B,K}$. If we set $\gamma_p\defeq\smallmat{\sqrt{\beta}}{-\sqrt{\beta}\bar\theta}{-1}{\theta}$, then $\rho_{k,p}(g)=\rho_k\bigl(\gamma_p i_p(g)\gamma_p^{-1}\bigr)$ (recall that $i_p$ is defined in \S\ref{sec2.1}). Also, $\rho_{k,\infty}$ and $\rho_{k,p}$ are compatible in the sense that they are equal when evaluated at elements of $\cB^\times$. If $A$ is a $K$-algebra, then $\rho_{k,\infty}(a)$ belongs to $\End_K\bigl(\mathcal{L}_r(K)\bigr)$ for $a\in\cB^\times$; in this case the map $\Phi\mapsto\widehat\Phi$, where 
\[ \widehat\Phi(g)\defeq\rho_k(\gamma_p^{-1})\rho_{k,p}(g_p^{-1})\Phi(g), \]
establishes a bijection between $A$-valued modular forms as in \eqref{CVMF} and $p$-adic modular forms as in \eqref{padicMF}. 
We call $\widehat\Phi$ the \emph{$p$-adic avatar} of $\Phi$. Finally, the map $\Phi\mapsto\bigl(g\mapsto\rho_k(g_p)\Phi(g)\bigr)$ gives a bijection between the $A$-module of $p$-adic quaternionic modular forms and the module $S_k(M,D,A)$ introduced in Definition \ref{defMF}. 

\subsubsection{Quaternionic eigenforms} \label{sec:quateigen} \label{quaternionic-subsubsec}

If $A$ is a ring, then the $A$-module $S_k(M,D;A)$ is equipped with an action of the (abstract) Hecke algebra that is generated over $A$ by Hecke operators $T_q$ for $q\nmid MD$ and $U_q$ for $q\,|\,MD$; these operators are defined, as customary, using suitable double coset decompositions (see, \emph{e.g.}, \cite[\S 2.3.5]{hida-iwasawa}). We denote by $\T_{M,D}^{(k)}(A)$ the image of this Hecke algebra inside $\End_A\bigl(S_k(M,D;A)\bigr)$. 

Notation being as in \S \ref{notation-subsec}, let $R$ stand for either $\cO_\wp$ or $\cO/\wp$; set $\T_{M,D}^{(k)}\defeq\T_{M,D}^{(k)}(\cO_\wp)$.  

 \begin{definition} \label{modforms}
 A \emph{quaternionic eigenform} of weight $k$ and level $M$ with coefficients in $R$ on $\cB$ is an $\cO_\wp$-algebra homomorphism $\phi:\T_{M,D}^{(k)}\rightarrow R$. 
 \end{definition} 
 
We denote by $\mathcal{S}_k(M,D;R)$ the set of all quaternionic eigenforms of weight $k$, level $M$, coefficients in $R$ on $\cB$. Suppose that $\mathcal{O}_\wp$ is an $\cO_{K_\p}$-algebra; then duality results for Hecke algebras show that for each $\phi\in \mathcal{S}_k(M,D;R)$ there is $\Phi\in S_k(M,D;R)$ that is an eigenform for all $T\in\T_{M,D}^{(k)}$ and satisfies $\phi(T)=a_T$, where $T(\Phi)=a_T\Phi$. 

A result of Deligne--Serre (\cite{DS}) ensures that every eigenform $\phi\in \mathcal{S}_k(M,D;\cO/\wp)$ is \emph{liftable to characteristic $0$}, \emph{i.e.}, there exists an eigenform $\Phi\in\mathcal S_k(M,D;\cO_\wp)$ such that $\phi(T)=\overline{\Phi(T)}$ for all Hecke operators $T\in\T_{M,D}^{(k)}$, where $x\mapsto \bar{x}$ is the projection map $\cO_\wp\twoheadrightarrow\cO/\wp$. 

If $D=1$, then the $\C$-vector space $H^1_\mathrm{par}\bigl(X_0(M),\mathcal L_r(\C)\bigr)$ is Hecke-equivariantly isomorphic to the $\C$-vector space $S_k(\Gamma_0(M))$ of modular forms of weight $k$ and level $\Gamma_0(M)$; in accordance with Definition \ref{modforms}, eigenforms of level $M$ on $\M_2(\Q)$ correspond to eigenforms in $S_k(\Gamma_0(M))$.  

\subsection{Jacquet--Langlands correspondence} \label{secJL} 

As before, let $\cO$ be the ring of integers of a number field. Consider the $\cO$-module $S_k^\text{$D$-new}\bigl(\Gamma_0(MD),\cO\bigr)\defeq S_k^\text{$D$-new}\bigl(\Gamma_0(MD),\Z\bigr)\otimes_\Z\cO$ of forms that are new at all the primes dividing $D$ and have Fourier coefficients in $\cO$. By the Jacquet--Langlands correspondence, a normalised eigenform $f\in S_k^\text{$D$-new}\bigl(\Gamma_0(MD),\mathcal{O}\bigr)$ is associated with a quaternionic eigenform 
\[ \JL(f):\T_{M,D}^{(k)}\longrightarrow\cO. \]
The map $f\mapsto\JL(f)$ establishes a bijection between the set of normalised eigenforms in $S_k^\text{$D$-new}\bigl(\Gamma_0(MD),\cO\bigr)$ and $\mathcal S_k(M,D;\cO)$. See, \emph{e.g.}, \cite[\S 2.3.6]{hida-iwasawa} for more details. 

\subsection{Galois representations} \label{galois-subsec}

Fix a prime $p\nmid MD$ and let $R$ be a complete noetherian local ring with finite residue field $k_R$ of characteristic $p$. Let $\phi\in\mathcal S_k(M,D;R)$ and denote by $\bar{\phi}\in\mathcal S_k(M,D;k_R)$ the composition of $\phi$ with the canonical reduction map $R\twoheadrightarrow k_R$. In light of the Jacquet--Langlands correspondence, to $\bar\phi$ one can attach a Galois representation 
\[ \rho_{\bar\phi}:G_\Q\longrightarrow \GL_2(k_R) \] 
unramified outside $MDp$ and such that if $\Frob_\ell\in G_\Q$ is a geometric Frobenius element at a prime $\ell\nmid MDp$, then the characteristic polynomial of $\rho_{\bar\phi}(\Frob_\ell)$ is $X^2-\bar{\phi}(T_\ell)X+\ell^{k-1}$ (see, \emph{e.g.}, \cite[\S 2.2]{carayol-formes}). If $\rho_{\bar\phi}$ is (absolutely) irreducible, then $\rho_{\bar\phi}$ is characterised (up to equivalence) by this property (see, \emph{e.g.}, \cite[Lemme 3.2]{DS}); moreover, by \cite[Th\'eor\`eme 3]{carayol-formes} there exists a unique (up to equivalence) continuous representation 
\[ \rho_\phi:G_\Q\longrightarrow\GL_2(R) \]
unramified outside $MDp$ and such that the characteristic polynomial of $\rho_\phi(\Frob_\ell)$ at a prime $\ell\nmid MDp$ is $X^2-\phi(T_\ell)X+\ell^{k-1}$. In other words, the reduction $\bar\rho_\phi$ of $\rho_\phi$ is (equivalent to) $\rho_{\bar\phi}$; for notational simplicity, we identify $\bar\rho_\phi$ with $\rho_{\bar\phi}$. 

When $R=\cO_\wp$, let $T_{\phi,\wp}$ be a free $\cO_\wp$-module of rank $2$ affording $\rho_\phi$, let $T_{\phi,\wp}^\dagger\defeq T_{\phi,\wp}(k/2)$ be the self-dual twist of $T_{\phi,\wp}$ and define $\overline{T}_{\phi,\wp}^\dagger\defeq T^\dagger_{\phi,\wp}/\wp T^\dagger_{\phi,\wp}$. Furthermore, we set $A_{\phi,\wp}^\dagger\defeq\Hom_{\cO_\wp}\bigl(T_{\phi,\wp}^\dagger,F_\wp/\cO_\wp\bigr)$, where $F_\wp$ is the fraction field of $\cO_\wp$.   

For technical reasons (namely, our need to use results from \cite{wang}), we also assume 
\begin{enumerate}
\item[(H-$p$)] $p>k+1$.
\end{enumerate}
In the list of conditions below, set $\sqrt{p^*}\defeq(-1)^{\frac{p-1}{2}}p$.

\begin{assumption} \label{ass} 
The eigenform $\phi$ satisfies the following conditions: 
\begin{enumerate} 
\item $\phi$ is ordinary, \emph{i.e.}, $\phi(T_p)\in R^\times$;
\item the restriction of $\overline{T}_{\phi,\wp}^\dagger$ to $\Gal\bigl(\bar\Q/\Q(\sqrt{p^*})\bigr)$ is absolutely irreducible;
\item $\overline{T}_{\phi,\wp}^\dagger$ is ramified at all primes $q\,|\,M$ with $q\equiv 1\pmod{p}$;
\item $\overline{T}_{\phi,\wp}^\dagger$ is ramified at all primes $q\,|\,D$ with $q\equiv\pm 1\pmod{p}$; 
\item there is a prime $q\,|\,MD$ such that $\overline{T}_{\phi,\wp}^\dagger$ is ramified at $q$. 
\end{enumerate}
\end{assumption}

These conditions are a higher weight counterpart of analogous ones appearing in \cite{zhang-selmer}.

\subsection{Level raising} \label{levelraisingsec}

We first recall the notion of ``admissible prime'', originally introduced by Bertolini--Darmon in their fundamental work on the Iwasawa main conjecture for elliptic curves over anticyclotomic $\Z_p$-extensions (\cite[p. 18]{BD-IMC}); more precisely, we recall its extension to the case of higher weight modular forms provided by \cite[Definition 1.1]{ChHs2}. 

\subsubsection{Admissible primes} \label{admissible-subsubsec}

Let us fix a quaternionic eigenform $\phi:\T_{M,D}^{(k)}\rightarrow R$ in the sense of Definition \ref{modforms} and assume that $\phi(U_p)$ is invertible modulo $\wp$. For the rest of the paper, we also pick an imaginary quadratic field $K$ of discriminant $D_K$ coprime to $p$ and assume that the coprime integers $M$ and $D$ satisfy the following condition: all the primes dividing $M$ split in $K$ and all the primes dividing $D$ are inert in $K$.  
 
\begin{definition} \label{admissible-def}
A prime number $\ell$ is \emph{admissible for} $(\phi,\wp,K)$ if 
\begin{enumerate}
\item $\ell$ does not divide $MDp$;
\item $\ell$ is inert in $K$;
\item $p$ does not divide $\ell^2-1$;
\item $\wp$ divides the ideal generated by $\ell^{\frac{k}{2}}+\ell^{\frac{k-2}{2}}-\epsilon_\ell \phi(T_\ell)$ with $\epsilon_\ell\in\{\pm1\}$.
\end{enumerate}
\end{definition}

We write $\mathcal P(\phi,\wp,K)$ for the set of square-free products of admissible primes for $(\phi,\wp,K)$, with the convention that $1\in\mathcal P(\phi,\wp,K)$. If no ambiguity is likely to arise, an admissible prime for $(\phi,\wp,K)$ will be simply called \emph{admissible}. Having fixed $\phi$, $\wp$, $K$,  we set $\mathcal{P}\defeq\mathcal{P}(\phi,\wp,K)$ and decompose $\mathcal P$ as
\[ \mathcal{P}=\mathcal{P}^\mathrm{def}\cup\mathcal{P}^\mathrm{indef} \] 
by requiring that $S\in\mathcal{P}^\mathrm{def}$ (respectively, $S\in\mathcal{P}^\mathrm{indef}$) if and only if $\mu(DS)=-1$ (respectively, $\mu(DS)=1$), where $\mu$ is the M\"obius function. In other words, $S\in\mathcal{P}^\mathrm{def}$ (respectively, $S\in\mathcal{P}^\mathrm{indef}$) if and only if the square-free integer $DS$ is a product of an \emph{odd} (respectively, \emph{even}) number of primes.
 
\subsubsection{Level raising} 
 
Recall that here $R$ stands either for $\cO$ or for $\cO/\wp$. Let $\phi\in\mathcal S_k(M,D;R)$, let $S\in\mathcal P$ and take $\xi\in\mathcal S_k(M,DS;R)$.
 
\begin{definition} \label{congruent-def}
The eigenforms $\phi$ and $\xi$ are \emph{congruent modulo $\wp$} if 
\begin{itemize}
\item $\phi(T_\ell)\equiv \xi(T_\ell)\pmod{\wp}$ for all primes $\ell\nmid MDS$;
\item $\phi(U_\ell)=\xi(U_\ell)\pmod{\wp}$ for all primes $\ell\,|\,MD$.
\end{itemize} 
\end{definition} 
In the situation of the definition above, we write $\phi\equiv \xi\pmod{\wp}$. Analogous notation will be used for eigenforms $\phi\in\mathcal S_k(M,D;\cO)$ and $\xi\in\mathcal S_k(M,DS;\cO/\wp)$ satisfying the congruence relations in Definition \ref{congruent-def}.

The next result was essentially proved in \cite{wang}.

\begin{theorem}[Wang] \label{levelraising} 
Let $\phi\in \mathcal{S}_k(M,D;\cO_\wp)$ and $S\in \mathcal{P}(\phi,\wp,K)$. 
\begin{enumerate}\item There exists $\bar\phi_S\in \mathcal{S}_k(M,DS;\cO/\wp)$ such that $\phi\equiv\bar\phi_S\pmod{\wp}$ and $\bar\phi_S(U_\ell)=\epsilon_\ell \ell^\frac{k-2}{2}$ for all primes $\ell\,|\,S$. 
\item There exist a number field with ring of integers $\cO_S\supset\cO$, a prime $\wp_S$ of $\cO_S$ above $\wp$ and ${\phi}_S\in\mathcal S_k(M,DS;\cO_{S,\wp_S})$ such that $\phi\equiv {\phi}_S\pmod{\wp_S}$, where $\cO_{S,\wp_S}$ is the completion of $\cO_S$ at $\wp_S$.
\end{enumerate} 
\end{theorem}

\begin{proof}[Sketch of proof] Pick a prime number $\ell\,|\,S$. If $\cB$ is indefinite, then apply \cite[Theorem 2.9]{wang}, while if $\cB$ is definite, then see the proof of \cite[Theorem 2.12]{wang}; in both cases we obtain a quaternionic eigenform as in the statement, which also satisfies Assumption \ref{ass} (all conditions are clearly the same except those at $\ell$, but, since $\ell$ is admissible, we have $\ell\not\equiv \pm 1\pmod{p}$). Applying these results recursively for all primes dividing $S$ concludes the proof. \end{proof}

\begin{notation} \label{JL-notation}
Given $T\in\mathcal{P}^\mathrm{indef} $, from here on we set $f_T\defeq\JL^{-1}(\phi_T)$ with $\JL$ as in \S \ref{secJL} and $\phi_T$ as in Theorem \ref{levelraising}. 
\end{notation}

\subsection{Selmer groups} 

Let $\phi\in\mathcal{S}_k(M,D;\cO_\wp)$. Let $v$ be a prime of $K$, let $K_v$ be the completion of $K$ at $v$, let $G_{K_v}\defeq\Gal(\bar{K}_v/K_v)$ be the absolute Galois group of $K_v$ (which can be identified with the decomposition subgroup of $G_K$ at $v$) and let $I_v\subset G_{K_v}$ be the inertia subgroup of $G_{K_v}$. Fix $S\in \mathcal{P}=\mathcal{P}(\phi,\wp,K)$. Let us set either $\cM\defeq T_{\phi,\wp}^\dagger$ or $\cM\defeq A_{\phi,\wp}^\dagger[\wp]\simeq T^\dagger_{\phi,\wp}/\wp\, T^\dagger_{\phi,\wp}$ and keep Assumption \ref{ass} in force. Accordingly, we shall view $\cM$ either as an $\cO_\wp$-module or as an $\cO/\wp$-vector space: in what follows, the letter $\RR$ will stand either for $\cO_\wp$ or for $\cO/\wp$, leaving it to the context to make clear which ring we are working with.

\subsubsection{Primes not dividing $pDS$} 

Let $v\nmid DSp$ be a prime of $K$. The \emph{finite} local condition at $v$ is 
\[ H^1_\mathrm{fin}(K_v,\cM)\defeq H^1\bigl(G_{K_v}/I_v,\cM^{I_v}\bigr). \]
It is well known that $H^1_\mathrm{fin}\bigl(K_v,A_{\phi,\wp}^\dagger[\wp]\bigr)$ and $H^1_\mathrm{fin}\bigl(K_v,T_{\phi,\wp}^\dagger/\wp\, T_{\phi,\wp}^\dagger\bigr)$ are exact annihilators of each other under the local Tate pairing. 

\subsubsection{Primes dividing $p$} 

Let $v\,|\,p$ be a prime of $K$. Since $p$ is ordinary, the restriction to the decomposition group $G_{K_v}$ of the action of $G_{\Q}$ on the rank $2$ free $\RR$-module $\cM$ admits a $G_{K_v}$-stable submodule $\cM ^+$ and a short exact sequence
\[ 0\longrightarrow \cM ^+\longrightarrow \cM \longrightarrow \cM ^-\longrightarrow 0 \]
such that $\cM ^\pm$ have both rank $1$ over $R$. After fixing an $\RR$-basis $\{v_1,v_2\}$ of $\cM$ with $v_1\in\cM ^+$, there is an isomorphism 
\[ \cM \simeq \begin{pmatrix}\eta^{-1}\chi_\cyc^{k/2}&*\\0&\eta\chi_\cyc^{-\frac{k-2}{2}}\end{pmatrix} \]
for an unramified character $\eta:G_{K_v}\rightarrow \cO_\wp^\times$ such that $\eta(\Frob_p)=\alpha_p$; here $\Frob_p$ denotes an arithmetic Frobenius at $p$, $\alpha_p$ is the $p$-adic unit root of Hecke polynomial $X^2-\phi(T_p)X+p^{k-1}$ at $p$ and $\chi_\cyc:G_{\Q_p}\rightarrow\Z_p^\times$ is the $p$-adic cyclotomic character. Define 
\[ H^1_\ord(K_v,\cM )\defeq\im\Bigl(H^1(K_v,\cM^+)\longrightarrow H^1(K_v,\cM)\Bigr). \] 
It is well known that $H^1_\ord\bigl(K_v,A_{\phi,\wp}^\dagger\bigr)$ and $H^1_\ord\bigl(K_v,T_{\phi,\wp}^\dagger\bigr)$ are exact annihilators of each other under the local Tate pairing. 

\subsubsection{Primes dividing $D$} \label{local D} 

Let $\ell\,|\,D$ be a prime number; since $\ell$ is inert in $K$, there is a unique prime of $K$, still denoted by $\ell$, lying over it. The restriction of $\cM$ to $G_{\Q_\ell}$ is 
\[ {\cM|}_{G_{\Q_\ell}}\simeq \begin{pmatrix}\chi_\cyc&0\\0&\mathbf{1}\end{pmatrix}, \] 
where $\mathbf{1}$ is the trivial character (\cite{carayol}). It follows that there is a unique line $\cM^+\subset\cM$ on which $G_{\Q_p}$ acts either by $\chi_\cyc$ or by $\chi_\cyc\tau_\ell$, where $\tau_\ell$ is the non-trivial element of $\Gal(K_\ell/\Q_\ell)$. Now we define 
\[ H^1_\ord(K_v,\cM)\defeq\im\Bigl(H^1(K_v,\cM^+)\longrightarrow H^1(K_v,\cM)\Bigr). \]

\subsubsection{Primes dividing $S$}\label{local S} 

Let $\ell\,|\,S$ be a prime number; in particular, $\ell$ is inert in $K$ and we denote by the same symbol the unique prime of $K$ above $\ell$. Define the \emph{singular part} of $H^1(K_\ell,\cM)$ as 
\[ H^1_\mathrm{sing}(K_\ell,\cM)\defeq{H^1(I_{K_\ell},\cM)}^{G_{K_\ell}/I_{K_\ell}}, \]
where $K_\ell$ is the completion of $K$ at the unique prime above $\ell$ and $I_{K_\ell}$ is the inertia subgroup of the absolute Galois group $G_{K_\ell}$ of $K_\ell$. Write 
\[ \partial_\ell:H^1(K_\ell,\cM)\longrightarrow H^1_\mathrm{sing}(K_\ell,\cM) \] 
for the \emph{residue} (\emph{i.e.}, restriction) map and set $H^1_\fin(K_\ell,\cM)\defeq\ker(\partial_\ell)$. As explained in \cite[\S1.3]{ChHs2}, both $H^1_\mathrm{sing}(K_\ell,\cM)$ and $H^1_\fin(K_\ell,\cM)$ are free of rank $1$ over $\RR$ and there is a canonical splitting 
\begin{equation} \label{splitting-eq}
H^1(K_\ell,\cM)\simeq H^1_\fin(K_\ell,\cM)\oplus H^1_\sing(K_\ell,\cM). 
\end{equation}
Let 
\begin{equation} \label{v-ell-eq}
v_\ell:H^1(K_\ell,\cM)\longrightarrow H^1_\mathrm{fin}(K_\ell,\cM) 
\end{equation}
be the map induced by \eqref{splitting-eq}. Abusing notation, we also write $\partial_\ell$ and $v_\ell$ for the compositions of $\partial_\ell$ and $v_\ell$, respectively, with the localisation map $\mathrm{loc}_\ell:H^1(K,\cM)\rightarrow H^1(K_\ell,\cM)$. 

The groups $H^1_\mathrm{fin}(K_\ell,\cM)$ and $H^1_\mathrm{sing}(K_\ell,\cM)$ also admit the following description. 
There is a canonical splitting $\cM\simeq\cM^+\oplus\cM^-$ with $\cM^-\defeq(\Frob_\ell-\epsilon_\ell \ell)\cM$ (\emph{cf.} the proof of \cite[Lemma 1.5]{ChHs2}). The \emph{ordinary} local condition at $\ell$ is 
\begin{equation} \label{ordS}
H^1_\ord(K_\ell,\cM)\defeq H^1(K_\ell,\cM^+),
\end{equation}
which is maximal isotropic with respect to the local Tate pairing. Then there are isomorphisms
\begin{equation} \label{ordS-2}
H^1_\ord(K_\ell,\cM)\simeq H^1_\mathrm{sing}(K_\ell,\cM),\quad H^1_\mathrm{fin}(K_\ell,\cM)\simeq H^1(K_\ell,\cM^-). 
\end{equation}

\subsubsection{Selmer groups} 

For every prime $v$ of $K$, let 
\[ \loc_v:H^1(K,\cM)\longrightarrow H^1(K_v,\cM) \]
be the localisation map at $v$. The \emph{Selmer group} $\Sel_S(K,\cM)$ of $\cM$ over $K$ is the subgroup of $H^1(K,\cM)$ consisting of those elements $s$ such that  
\begin{itemize}
\item $\loc_v(s)\in H^1_\mathrm{fin}(K_v,\cM)$ if $v\nmid pDS$; 
\item $\loc_v(s)\in H^1_\mathrm{ord}(K_v,\cM)$ if $v\,|\,pDS$. 
\end{itemize}
When $S=1$, we simply write $\Sel(K,\cM)$ for $\Sel_1(K,\cM)$. Let $T$ be a product of distinct primes with $(T,D_KMDp)=1$. The \emph{relaxed $S$-Selmer group $\Sel_{S,(T)}(K,\cM)$ of $\cM$ at $T$ over $K$} is the subgroup of $H^1(K,\cM)$ consisting of those elements $s$ such that  
\begin{itemize}
\item $\loc_v(s)\in H^1_\mathrm{fin}(K_v,\cM)$ if $v\nmid pDST$;
\item $\loc_v(s)\in H^1_\ord(K_v,\cM)$ if $v\,|\,pDS$. 
\end{itemize}
For a (possibly infinite) set $\mathscr D$ of prime numbers, the \emph{relaxed $\mathscr D$-Selmer group $\Sel_{S,(\mathscr D)}(K,\cM)$ of $\cM$ at $T$ over $K$} is defined in a completely analogous way. We also denote by $\Sel_{S,[T]}(K,\cM)$ the subgroup of $\Sel_{S,(T)}(K,\cM)$ consisting of those classes $c$ such that 
\begin{itemize}
\item $\loc_v(c)\in H^1_\sing(K_v,\cM)$ for all $v\,|\,T$. 
\end{itemize}
For $T'$ a product of distinct primes with $(T',D_KMDTp)=1$, we define $\Sel_{S,(T),[T']}(K,V)$ to be the subgroup of $\Sel_{S,(TT')}(K,V)$ consisting of those classes $c$ such that 
\begin{itemize}
\item $\loc_v(c)\in H^1_\sing(K_v,V)$ for all $v\,|\,T'$. 
\end{itemize}
For every $\Gal(K/\Q)$-module $\cM$, we write $\cM^\pm$ for the $\pm1$-eigenspaces of $\cM$ under the action of the generator of $\Gal(K/\Q)$. In particular, since $p$ is odd, there is a splitting
\[ \Sel_{S}(K,\cM)=\Sel_{S}(K,\cM)^+\oplus\Sel_{S}(K,\cM)^-, \]
and analogously for the other Selmer groups.

\subsubsection{Comparison of Selmer groups} 
The notation introduced in Theorem \ref{levelraising} is in force. If $v$ is a prime of $K$, then let $K_v$ be the completion of $K$ at $v$, let $G_{K_v}\defeq\Gal(\bar{K}_v/K_v)$ be the absolute Galois group of $K_v$, which can be identified with the decomposition subgroup of $G_K$ at $v$, and let $I_v\subset G_{K_v}$ be the inertia subgroup of $G_{K_v}$. Fix $S\in \mathcal{P}=\mathcal{P}(\phi,\wp,K)$ and let $ \phi_S$ be as in Theorem \ref{levelraising}.
Set $k_\wp\defeq\cO/\wp$ and $k_{\wp_S}\defeq\cO_S/\wp_S$; by construction, $k_{\wp_S}$ is a finite extension of $k_\wp$. Note that $S$ appears as a divisor of the discriminant $DS$ of the quaternion algebra on which $\phi_S$ is defined; it appears also as an admissible integer for $\phi$. 

The next lemma gives a comparison between Selmer groups.

\begin{lemma} \label{lemma selmer} 											
$\Sel_S\bigl(K,A^\dagger_{\phi,\wp}[\wp]\bigr)\otimes_{k_\wp}k_{\wp_S}\simeq\Sel\bigl(K,A^\dagger_{ \phi_S,\wp_S}[\wp_S]\bigr)$. 
\end{lemma}

\begin{proof} By \eqref{ordS}, at primes dividing $S$ the ordinary local conditions for $\phi_S$ in \S\ref{local D} are the same 
as the ordinary local conditions for $\phi$ in \S\ref{local S}. Furthermore, there are isomorphisms of $k_{\wp_S}$-vector spaces $A_{\phi_S}[\wp_S]\simeq A_{\phi,\wp}\otimes_{k_\wp}k_{\wp_S}$ and
\[ \Sel_S\bigl(K,A^\dagger_{\phi,\wp}[\wp]\bigr)\otimes_{k_\wp}k_{\wp_S}\simeq\Sel_S\bigl(K,A^\dagger_{\phi,\wp}[\wp]\otimes_{k_\wp}k_{\wp_S}\bigr). \] 
The result follows. \end{proof}

Now we compare $\Sel_{{S}}\bigl(K,A_{\phi,\wp}^\dagger[\wp]\bigr)$ and $\Sel_{{S\ell}}\bigl(K,A_{\phi,\wp}^\dagger[\wp]\bigr)$, and also $\Sel_{{S}}\bigl(K,A_{\phi,\wp}^\dagger[\wp]\bigr)^\pm$ and $\Sel_{{S\ell}}\bigl(K,A_{\phi,\wp}^\dagger[\wp]\bigr)^\pm$, for an admissible prime $\ell\nmid S$.
  
\begin{lemma} \label{lemma7.2} 
Let $\ell\nmid S$ be an admissible prime. 
\begin{enumerate}
\item $\Sel_{{S\ell}}\bigl(K,A_{\phi,\wp}^\dagger[\wp]\bigr)^\pm=\ker\Bigl(\loc_\ell:\Sel_{{S}}\bigl(K,A_{\phi,\wp}^\dagger[\wp]\bigr)^\pm\rightarrow H^1_\fin\bigl(K_\ell,A_{\phi,\wp}^\dagger[\wp]\bigr)\!\Bigr)$.
\item $\dim_{k_\wp}\Sel_{{S\ell}}\bigl(K,A_{\phi,\wp}^\dagger[\wp]\bigr)^\pm\geq\dim_{k_\wp}\Sel_{{S}}\bigl(K,A_{\phi,\wp}^\dagger[\wp]\bigr)^\pm-1$. 
\item If $\loc_\ell:\Sel_S\bigl(K,A_{\phi,\wp}^\dagger[\wp]\bigr)^\pm\rightarrow H^1_\mathrm{fin}\bigl(K_\ell,A_{\phi,\wp}^\dagger[\wp]\bigr)$ is surjective, then
\[ \dim_{k_\wp}\Sel_{S\ell}\bigl(K,A_{\phi,\wp}^\dagger[\wp]\bigr)^\pm=\dim_{k_\wp}\Sel_S\bigl(K,A_{\phi,\wp}^\dagger[\wp]\bigr)^\pm-1. \] 
\end{enumerate}
Analogous results hold true for $\Sel_{{S}}\bigl(K,A_{\phi,\wp}^\dagger[\wp]\bigr)$ and $\Sel_{{S\ell}}\bigl(K,A_{\phi,\wp}^\dagger[\wp]\bigr)$.
\end{lemma} 

\begin{proof} Since $\dim_{k_\wp}H^1_\fin\bigl(K_\ell,A_{\phi,\wp}^\dagger[\wp]\bigr)=1$, it suffices to show (1), \emph{i.e.}, that $\Sel_{S\ell}\bigl(K,A_{\phi,\wp}^\dagger[\wp]\bigr)^\pm$ is the kernel of the localisation map. The kernel of $\loc_\ell$ consists of those $s\in \Sel_{S}\bigl(K,A_{\phi,\wp}^\dagger[\wp]\bigr)^\pm$ such that the restriction of $\loc_\ell(s)$ belongs to $H^1_\sing\bigl(K_\ell,A_{\phi,\wp}^\dagger[\wp]\bigr)$, which is isomorphic to $H^1_\ord\bigl(K_\ell,A_{\phi,\wp}^\dagger[\wp]\bigr)$ by \eqref{ordS-2}, and the claim follows. The statements for $\Sel_{{S}}\bigl(K,A_{\phi,\wp}^\dagger[\wp]\bigr)$ and $\Sel_{{S\ell}}\bigl(K,A_{\phi,\wp}^\dagger[\wp]\bigr)$ can be proved in the same way. \end{proof}



\subsection{Special points}  \label{secspecialpoints}

Here we introduce special points both in the indefinite case and in the definite case. Let us fix, as before, $M$, $D$, $p$ and recall that $p\nmid MD$. Let $K$ be an imaginary quadratic field where all the primes dividing $M$ split and all the primes dividing $D$ are inert. Define the set $\mathcal{S}_{M,D}(K)$ of \emph{special points} of $Y_0(M,D)$ as 
\[ \mathcal{S}_{M,D}(K)\defeq\hat{R}^\times\big\backslash\bigl(\hat{B}^\times\times\Hom(K,B)\bigr)\big/B^\times. \] 
By extending scalar, we view $\mathcal{S}_{M,D}(K)$ as a subset of $Y_0(M,D)(\C)$. Let $\cO_K$ be the ring of integers of $K$ and for any $c\geq1$ let $\cO_c\defeq\Z+c\cO_K$ be the order of $K$ of conductor $c$. If $x\in\mathcal{S}_{M,D}(K)$ is represented by a pair $(g,\varphi)$ such that $\varphi(K)\cap g^{-1}\hat{R}^\times g=\varphi(\mathcal{O}_c)$, then we say that $x$ has \emph{conductor $c$}. In the indefinite (respectively, definite) case, the special points in 
$\mathcal{S}_{M,D}(K)$ are called \emph{Heegner} (respectively, \emph{Gross}) \emph{points}. 

In both cases, $\mathcal{S}_{M,D}(K)$ is equipped with an action of $\Gal(K^\mathrm{ab}/K)$, where $K^\mathrm{ab}$ is the maximal abelian extension of $K$, that is defined as follows: if $\sigma\in\Gal(K^\mathrm{ab}/K)$ is represented by $\mathfrak{a}\in\hat{K}^\times$ via the geometrically normalised Artin map and $x$ is represented by $(g,\varphi)$, then $x^\sigma\in\mathcal{S}_{M,D}(K)$ is the point of $\mathcal{S}_{M,D}(K)$ represented by the pair $\bigl(g\hat{\varphi}(\mathfrak{a}),\varphi\bigr)$, where $\hat{\varphi}:\hat K\rightarrow \hat{B}$ is the map obtained from $\varphi$ by extending scalars to $\hat{\Z}$. In the indefinite case, this action coincides with the usual Galois actions on points, by Shimura's reciprocity law; furthermore, Heegner points of conductor $c$ are rational over the ring class field $H_c$ of $K$ of conductor $c$.

Following \cite[\S2.2]{ChHs2} and \cite[\S3.1]{wang} (see also \cite{CKL}, \cite{CL}), we make a specific choice of special points. 
As before, write $G$ for the algebraic group defined by $\cB^\times$; fix $c\geq 1$ with $(c,MD)=1$ and for each prime number $q$ define $\varsigma_q(c)\in G(\Q_q)$ by setting 
\[\varsigma_q(c)\defeq\begin{cases} 1 & \text{if $q\nmid Mc$}, \\[2mm]
 \delta_K^{-1}\smallmat{\theta}{\bar\theta}{1}{1} & \text{if $q\,|\,M$},\\[2mm]
\smallmat{q^n}{0}{0}{1} & \text{if $q\,|\,c$ splits in $K$ and $n=\val_q(c)$},\\[2mm]
\smallmat{1}{q^{-n}}{0}{1} & \text{if $q\,|\,c$ is inert in $K$ and $n=\val_q(c)$}.
\end{cases}\]
Set also $\varsigma(c)\defeq\bigl(\varsigma_q(c)\bigr)_q\in G\bigl(\mathds{A}^{(\infty)}\bigr)$. For any place $v$ of $\Q$, the \emph{Atkin--Lehner involution} $\tau_v\in G(\Q_v)$ is 
\[ \tau_v\defeq\begin{cases}
\smallmat{0}{1}{-M}{0} & \text{if $v\,|\,M$},\\[2mm]
J_\cB & \text{if $v\,|\,\infty D$},\\[2mm]
1 & \text{if $v\nmid MD$}.
\end{cases} \]
Finally, put $\tau_\cB\defeq{(\tau_v)}_v\in G(\mathds{A})$; with an abuse of notation, we sometimes denote by $\tau_\cB$ the element $(\tau_q{)}_q\in G\bigl(\mathds{A}^{(\infty)}\bigr)$. 

\subsection{CM elliptic curves} \label{CM-subsec}

Let $E$ be a fixed elliptic curve with CM by $\cO_K$, defined over the Hilbert class field $H_K$ of $K$; we also fix a complex analytic isomorphism $E(\C)\simeq \C/\cO_K$, which we will treat as an equality. 

Let $\mathcal{C}_{c}\defeq c^{-1}\mathcal{O}_{c}/\cO_K\simeq \Z/c\Z$ and $E_c\defeq E/\mathcal{C}_c$; then $E_c$ is an elliptic curve with CM by $\mathcal{O}_c$ equipped with a complex analytic uniformization $E_c(\C)\simeq\C/\cO_c$ such that the quotient map $\varphi_{c}:E\twoheadrightarrow E_c$ is the isogeny given by $[z]\mapsto [cz]$ as a map of complex tori. 

Fix a fractional $\cO_{c}$-ideal $\mathfrak{a}$, let $E_c[\mathfrak{a}]\subset E_c(\C)$ be the $\mathfrak{a}$-torsion subgroup of $E_c(\C)$ and form the quotient $E_\mathfrak{a}\defeq E_c/E_c[\mathfrak{a}]$; then $E_c[\mathfrak{a}]\simeq \mathfrak{a}^{-1}\cO_c/\cO_{c}$ and there is a complex analytic isomorphism $E_\mathfrak{a}(\C)\simeq \C/\mathfrak{a}^{-1}$. Finally, there is a canonical quotient map $\lambda_\mathfrak{a}:E_c\twoheadrightarrow E_\mathfrak{a}$ that, as a map of complex tori, is induced by the inclusion $\cO_c\subset\mathfrak{a}^{-1}$. 

\subsection{Heegner cycles}\label{HC} \label{cycles-subsec} 

Suppose we are in the \emph{indefinite} case. We require the isomorphisms $i_\infty$ in \eqref{iinfty} to further satisfy the condition   
$i_\infty(\theta)=\smallmat{\theta+\bar\theta}{-\theta\bar\theta}{1}{0}$. The action of $i_\infty(K)\subset\GL_2(\R)$ has two fixed points, say $\omega$ and $\bar\omega$, in $\mathcal H\cup\boldsymbol\rho(\mathcal H)$, where $\boldsymbol\rho$ is complex conjugation. Suppose that $\omega\in\mathcal H$ and $\bar\omega=\boldsymbol\rho(\omega)$. Let $c\geq 1$ be an integer prime to $MD$. Using the first bijection in \eqref{bijections-eq}, the \emph{Heegner point} of conductor $c$ is $P_c\defeq\bigl[\bigl(\varsigma(c)\tau_\cB,\omega\bigr)\bigr]\in Y_0(M,D)(\C)$. Then the Galois action on such points is given by $P_c^{\sigma_\mathfrak{a}}\defeq\bigl[\bigl(a\varsigma(c)\tau_\cB,\omega\bigr)\bigr]$ (as in \S\ref{sec2.1}, we write just $a$ for $\psi_K(a)$). Let $\widetilde{P}_c$ be a lift of $P_c$ to ${X}_d(M,D)$. 

Heegner cycles in ${X}_d(M,D)$ are defined as certain cycles fibered over $\widetilde{P}_c$; when $D=1$ we follow the  construction in \cite{Nek} and \cite{LV}, while when $D>1$ we use alternative constructions in \cite{BesserCM}, \cite{Chida}, \cite{E-dVP}. Let $D_c=c^2D_K$ be the conductor of the order $\cO_c$. To simplify the notation, set $r\defeq k-2$ and $u\defeq r/2$; also, let $W_{k,d}\defeq W_{k,d}(M,D)$, $\mathscr{A}_d\defeq A_d(M,D)$ and $X_d\defeq X_d(M,D)$; for $D=1$ we also write $\mathscr{E}_d$ for $\mathscr{A}_d$, to emphasize that in this case the universal object is an elliptic curve. We use the symbol $i_{\widetilde{P}_\mathfrak{a}}$ for the inclusion of the fiber over $\widetilde{P}_\mathfrak{a}$ in the universal object; in other words, $i_{\widetilde{P}_\mathfrak{a}}:E_\mathfrak{a}\hookrightarrow \mathscr{E}_d$ for $D=1$ and $i_{\widetilde{P}_\mathfrak{a}}:A_\mathfrak{a}=E_\mathfrak{a}\times E\hookrightarrow \mathscr{A}_d$ for $D>1$. 

The Kuga--Sato variety $W_{k,d}$ can be equipped with several projectors; in particular, there is a projector $\epsilon_W$ inducing isomorphisms 
\begin{equation}\label{projectors} \epsilon_W H^*_\et\bigl(\overline{W}_{k,d},\Z_p\bigr)\overset\simeq\longrightarrow \epsilon_W H^{k-1}_\et\bigl(\overline{W}_{k,d},\Z_p\bigr)\overset\simeq\longrightarrow H^1_\et\bigr(\overline{X}_d,j_*\Sym^r(\R^1\pi_{d*}\Z_p)\bigr), 
\end{equation}
where if $D>1$, then $\pi_d:A_d(M,D)\rightarrow X_d(M,D)$ is the universal abelian surface and $j$ is the identity, while if $D=1$, then $\pi_d:E_{d}(M)\rightarrow Y_d(M)$ is the universal elliptic curve and $j:Y_d(M)\hookrightarrow X_d(M)$ is the canonical inclusion (see, \emph{e.g.}, \cite[Lemma 2.2]{wang} and the references in \cite{wang}). Let $\epsilon_d$ be the projector associated with the covering map $X_d(M,D)\rightarrow X_0(M,D)$ (\emph{cf.} \cite[Lemma 2.1]{wang}). We fix a prime number $p$ such that 
\begin{itemize}
\item $p\nmid MDdc\phi(d)(k-2)!$. 
\end{itemize}
For a variety $V$ defined over a number field $L$, write $\mathrm{CH}^{i}_0(V/L)$ for the abelian group of homologically trivial cycles of codimension $i$ in $V$ that are rational over $L$, modulo rational equivalence; moreover, for a ring $R$ set $\mathrm{CH}^{i}_0(V/L{)}_R\defeq\mathrm{CH}^{i}_0(V/L)\otimes_\Z R$. 

%
%
%

In the notation of \S\ref{sec2.1.1}, if $D=1$, then $\widetilde{P}_c$ is represented by a triple $(E_c,C_c,\nu_d)$. For each fractional ideal $\mathfrak{a}$, let $\sigma_\mathfrak{a}\in\Gal(H_c/K)\simeq\Pic(\mathcal{O}_c)$ be the element corresponding to $\mathfrak{a}$ via the geometrically normalised Artin map; then $\widetilde{P}_c^{\sigma_\mathfrak{a}}$ is represented by a triple $\bigl(E_\mathfrak{a},C_c^{(\mathfrak{a})},\nu_d^{(\mathfrak{a})}\bigr)$ (and one can check that $C_c^{(\mathfrak{a})} =\lambda_\mathfrak{a}(C_c)$ and $\nu_d^{(\mathfrak{a})}=\lambda_\mathfrak{a}\circ\nu_d$, which we do not need in the following; see, \emph{e.g.}, \cite[\S2.4]{CH}). 
Let
\[ \Gamma_\mathfrak{a}\defeq\Bigl\{\bigl(\sqrt{-D_c}\cdot z,z\bigr)\;\Big|\;z\in E_\mathfrak{a}\Bigr\} \]
be the transpose of the graph of $\sqrt{-D_c}$ in $E_\mathfrak{a}\times E_\mathfrak{a}$.  

When $D>1$, $\widetilde{P}_c$ is represented by the quadruplet $(E_c\times E_c,\iota_c ,C_c,\nu_c)$. For each fractional $\mathcal{O}_c$-ideal $\mathfrak{b}$, the point $\widetilde{P}_c^{\sigma_\mathfrak{b}}$ is represented by a quadruplet $\bigl(E_\mathfrak{a}\times E_c,\iota^{(\mathfrak{a})},C_c^{(\mathfrak{a})},\nu_d^{(\mathfrak{a})}\bigr)$ for a unique fractional $\mathcal{O}_c$-ideal $\mathfrak{a}$ (the precise relation between $\mathfrak{a}$ and $\mathfrak{b}$, which is not needed in this paper, can be deduced from, \emph{e.g.}, \cite[\S2.3]{E-dVP}).  Let
\[ \Gamma_\mathfrak{a}\defeq\Bigl\{\bigl(\sqrt{-D_c}\cdot\lambda_\mathfrak{a}(z),z\bigr)\;\Big|\; z\in E_c\Bigr\} \]
be the transpose of the graph of $\sqrt{-D_c}\circ\lambda_{\mathfrak a}:E_c\rightarrow E_\mathfrak{a}$ in $E_\mathfrak{a}\times E_c$. 

Now let $D\geq1$. Set $\Upsilon_\mathfrak{a}\defeq i_{\widetilde{P}_c}(\Gamma_\mathfrak{a}^u)\subset W_{k,d}$. 

\begin{definition} \label{defHeegnercycle} 
The \emph{Heegner cycle} attached to  $\mathfrak{a}$ is $\Delta_\mathfrak{a}\defeq\epsilon_d\epsilon_W\Upsilon_\mathfrak{a}$.
\end{definition}

Under our assumptions on $p$, the cycle $\Delta_\mathfrak{a}$ belongs to $\epsilon_d\epsilon_W\mathrm{CH}^{k/2}_0(W_{k,d}/H_c)_{\Z_p}$  
by \eqref{projectors}. Set $\Delta_c\defeq\Delta_{\cO_c}$; when we need to specify $M$ and $D$, we write $\Delta_{\mathfrak{a},M,D}$ for $\Delta_\mathfrak{a}$.

\begin{remark} 
For any $D\geq 1$, observe that $\Delta_c$ is unchanged if $\Gamma_c$ is replaced by the divisor $\Gamma_c-(E_c\times\{0\})-cD_K(\{0\}\times E_c)$ (\emph{cf.} \cite[Ch. 2, (3.6)]{Nek2} for $D=1$ and \cite[Proposition 4.2]{E-dVP} for $D>1$). In particular, since $p\nmid c$, Definition \ref{defHeegnercycle} is consistent (up to $p$-adic units, due to a further normalisation in terms of certain self-intersection properties) with that of \cite[\S3.1]{wang}.  
\end{remark}

Let us fix a quaternionic eigenform $\phi\in\mathcal{S}_k(M,D;\cO_\wp)$ as in \S \ref{quaternionic-subsubsec}, where $\wp$ is a fixed prime ideal of $\cO$ of residual characteristic $p$. Taking the $\phi$-eigenspace in the cohomology group with coefficients in $\cO_\wp$ (which can be done thanks to \cite[Lemma 2.2]{Nek}), for any number field $L$ we obtain a $\wp$-adic Abel--Jacobi map 
\begin{equation}\label{AJ1}
\AJ_{\phi,\wp,L}:\epsilon_d\epsilon_W\CH^{k/2}_0(W_{k,d}\otimes_\Q L{)}_{\Z_p}\longrightarrow H^1\bigl(L,T_{\phi,\wp}^\dagger\bigr).
\end{equation}

\begin{definition}\label{defHC} 
The \emph{Heegner cycle} attached to $\mathfrak{a}$ and $\phi$ is $y _{\mathfrak{a},M,D}(\phi)\defeq\AJ_{\phi,\wp,H_c}\bigl(\Delta_{\mathfrak{a},M,D} \bigr)$.
\end{definition}

When $\mathfrak{a}=\cO_c$, we will write $y _{c,M,D}(\phi)$ in place of $y _{\mathfrak{a},M,D}(\phi)$. 

\begin{remark} \label{selmer-rem}
As explained in \cite[p. 2332]{wang}, the map $\AJ_{\phi,\wp,K}$ factors through $\Sel\bigl(K,T_{\phi,\wp}^\dagger\bigr)$. 
\end{remark}

\subsection{Theta elements} 

In the definite case, special cycles give rise to theta elements and, eventually, to anticyclotomic $p$-adic $L$-functions; we quickly review the subject as developed in \cite{CL}, \cite{ChHs2}, \cite{ChHs1}. Let $\phi\in \mathcal{S}_k(M,D;R)$ with $R$ as before. We assume that $\phi$ is \emph{primitive}, \emph{i.e.}, there is no $\phi'\in \mathcal{S}_k(M,D;R)$ such that $\phi=\wp\phi'$. We write $\Phi$ for the modular form associated with $\phi$, as discussed in \S\ref{sec:quateigen}. Let $c\geq1$ be an integer prime to $MDp$. Recall the notation in \S\ref{secspecialpoints} and define the \emph{Gross point} of conductor $c$ to be $P_c\defeq\bigl[\bigl(\varsigma(c),\psi_K\bigr)\bigr]$. Set $\mathcal{G}_c\defeq\Gal(H_c/K)$ and recall the element $\gamma_\p$ introduced in \S \ref{definite-subsubsec}. Set $\mathbf{v}_0^*\defeq D_K^{\frac{k-2}{2}}\cdot\mathbf{v}_0$. The theta element associated with $\phi$ is 
\[ \theta_{c,M,D}(\phi)\defeq\sum_{\sigma\in\mathcal{G}_c}\big\langle\rho_k^{-1}(\gamma_\p)\mathbf{v}_0^*,\Phi(P_c^\sigma)\big\rangle_k\cdot\sigma\in R[\mathcal{G}_c]. \] 
When $A$ is a $K$-algebra, we may compare these elements with the theta elements introduced in \cite[\S4.3]{ChHs1} (\emph{cf.} also \cite[Definition 2.3]{CL}). Let $\Psi$ be a modular form as in \eqref{CVMF} and let $\Phi$ be the modular form as in Definition \ref{defMF} corresponding to $\Psi$. The theta elements in \cite{ChHs1} are defined as 
\begin{equation} \label{CHtheta}
\sum_{\sigma\in\mathcal{G}_{c}}\big\langle\mathbf{v}_0^*,\Psi(P_c^\sigma)\big\rangle_k\cdot\sigma\in R[\mathcal G_c].
\end{equation}
Now recall the $p$-adic avatar $\widehat\Psi(g)=\rho_k^{-1}(\gamma_\p)\rho_{k,p}^{-1}(g_p)\Psi(g)$ from \cite[\S4.1]{ChHs1}; using this formula and the equality $\rho_{k,p}(g)=\rho_k(\gamma_\p g_p\gamma_\p^{-1})$, we get 
 \[ \Psi(g)=\rho_k(\gamma_\p) \rho_k(g_p)\hat\Psi(g)=\rho_k(\gamma_\p) \Phi(g). \]
Replacing in \eqref{CHtheta}, we finally get
\[ \sum_{\sigma\in\mathcal{G}_{c}}\big\langle\mathbf{v}_0^*,\Psi(P_c^\sigma)\big\rangle_k\cdot\sigma=\sum_{\sigma\in\mathcal{G}_{c}}\big\langle\rho_k^{-1}(\gamma_\p) (\mathbf{v}_0^*),\Phi\bigl(x_n(a)\bigr)\big\rangle_k\cdot\sigma=\theta_{c,M,D}(\phi). \]
The interpolation properties of the elements $\theta_{c,M,D}(\phi)$ can be found in \cite[Theorem 4.6]{ChHs1} (\emph{cf.} also \cite[Theorem 2.4]{CL}); a more general interpolation formula is due to Hung (\cite{Hung}). More precisely, let $f\in S_k(\Gamma_0(MD))$ be a newform and let $\phi_f$ be the associated quaternionic eigenform (see the discussion in \S\ref{secJL}). Let $L_K(f,\chi,s)$ be the $L$-function of $f$ twisted by $\chi$, and let $\Omega_{f,D}$ be the \emph{Gross period} of $f$ (see \cite[\S1]{ChHs1}; see also \cite[\S5.2]{KL1}). 
For all (finite order) characters $\chi$ of $\mathcal{G}_c$ we have 
\[ \chi\bigl(\theta_{c,M,D}(\phi_f)^2\bigr)=\mathrm{Const}\times  \frac{L_K(f,\chi,k/2)}{\Omega_{f,D}}, \]
where $\mathrm{Const}$ is a $p$-adic unit (by which we mean that there is a finite field extension $\mathcal K$ of $\Q_p$ such that $\mathrm{Const}\in\cO_{\mathcal K}^\times$, where $\cO_{\mathcal K}$ is the valuation ring of $\mathcal K$). 

\begin{remark}
In the interpolation formulas above it is implicit that the quotient between the special value of the complex $L$-function and the Gross period is an algebraic number, so it may be compared via the given embeddings $\iota_\infty:\bar\Q\hookrightarrow\C$ and $\iota_p:\bar\Q\hookrightarrow\bar\Q_p$ from \S \ref{notation-subsec}  to the value of the theta series at $\chi$. 
\end{remark}

\subsection{Reciprocity laws} \label{reciprocity-subsec}

We review the explicit reciprocity laws of \cite{wang}, which generalise those of \cite{BD-IMC} and \cite{ChHs2} and lie at the heart of this paper. Let us choose $\phi\in \mathcal{S}_k(M,D;R)$ with $R\defeq\cO_\wp/\wp\cO_\wp=\cO/\wp$, where $\cO$ is the ring of integers of a number field, $\wp$ is a prime ideal of $\cO$ of residue characteristic $p$ and $\cO_\wp$ is the completion of $\cO$ at $\wp$. We suppose throughout that Assumption \ref{ass} is satisfied and that $\phi$ is normalised. Let $\mathcal{P}=\mathcal{P}(\phi,\wp,K)$ be as in \S \ref{admissible-subsubsec} and let $\phi_{S}\in\mathcal{S}_k(M,DS;R)$ be the form from Theorem \ref{levelraising}. 

Let $c\geq1$ be an integer prime to $MD$. If $S\in\mathcal{P}^\mathrm{def}$, then we set 
\[ \lambda_{\phi,c}(S)\defeq\theta_{c,M,DS}(\phi_S). \]
If $S\in \mathcal{P}^\mathrm{indef}$, then we set 
\[ \kappa_{\phi,c}{(S)}\defeq\sum_{\sigma\in\mathcal{G}_c}\bigl(y _{c,M,DS}(\phi_S)\bigr)^\sigma\in H^1\bigl(H_c,T_{\phi_{S},\wp_S}^\dagger\bigr)\simeq H^1\bigl(H_c,T_{\phi,\wp}^\dagger\bigr), \]
where the isomorphism on the right is a consequence of the equivalence $T_{\phi_{S},\wp_S}\simeq T_{\phi.\wp}$. 

For $\ell\in\mathcal P$, recall the map $\partial_\ell$, whose target $H^1_\mathrm{sing}\bigl(K_\ell,T_{\phi,\wp}^\dagger\bigr)$ is a free $R$-module of rank $1$. 

\begin{theorem}[First reciprocity law] \label{ERL1}
Let $S\in\mathcal{P}^\mathrm{def}$ and $\ell\nmid S$ be an admissible prime for $(\phi,\wp,K)$. Then 
\[ \partial_\ell\bigl(\kappa_{\phi,c}{(S\ell)}\bigr)\equiv u\lambda_{\phi,c}{(S)}\pmod{\wp} \] 
for some $u\in R^\times$. 
\end{theorem} 

\begin{proof} This is proved in \cite[Theorem 3.6]{wang}; see also \cite[\S 8.2]{Chida}. {\color{red}}\end{proof}

\begin{theorem}[Second reciprocity law] \label{ERL2} 
Let $S\in\mathcal{P}^\mathrm{indef}$ and let $\ell\nmid S$ be an admissible prime for $(\phi,\wp,K)$. Then
\[ v_\ell\bigl(\kappa_{\phi,c}{(S)}\bigr)\equiv u\lambda_{\phi,c}{(S\ell)}\pmod{\wp} \] 
for some $u\in R^\times$. 
\end{theorem}

\begin{proof} This is proved in \cite[Theorem 3.4]{wang}.\end{proof}

For $\ell\in\mathcal{P}$, recall the map $v_\ell$ from \eqref{v-ell-eq}, whose target $H^1_\mathrm{fin}\bigl(K_\ell,T_{\phi,\wp}^\dagger\bigr)$ is a free $R$-module of rank $1$. 

\begin{corollary} \label{coro2.5} 
Let $S\in\mathcal{P}^\mathrm{indef}$ and let $\ell_1$, $\ell_2$ be distinct admissible primes not dividing $S$. Then 
\[ v_{\ell_1}\bigl(\kappa_{\phi,c}(S)\bigr)\neq 0\Longleftrightarrow\partial_{\ell_2}\bigl(\kappa_{\phi,c}(S\ell_1\ell_2)\bigr)\neq 0. \] 
\end{corollary}

\begin{proof} This follows immediately upon combining Theorems \ref{ERL1} and \ref{ERL2}. \end{proof}

\begin{remark} 
It would be interesting to explore arithmetic applications of (more general versions of) the explicit reciprocity laws for all integers $c\geq1$, with possible applications to the Tamagawa number conjecture for the motive of $f$; results in this direction can be found in \cite{Chida} and \cite{LV-JNT}. 
\end{remark}

\section{Kolyvagin's conjecture} \label{kolyvagin-sec}

In this section, we introduce our Kolyvagin systems of derived Galois cohomology classes attached to a higher weight newform $f$ and then prove Kolyvagin's conjecture for $f$.

\subsection{The newform $f$ and its Galois representations}

Let $f\in S_k(\Gamma_0(N))$ be a newform of weight $k\geq4$ and level $N$, whose $q$-expansion will be denoted by $f(q)=\sum_{n\geq 1}a_n(f)q^n$. Write $F\defeq\Q\bigl(a_n(f)\mid n\geq1\bigr)$ for the Hecke field of $f$, which is a totally real number field, and let $\cO_F$ be its ring of integers. Furthermore, let $\p$ be a prime ideal of $\cO_F$ of  residual characteristic $p\nmid N$. Let $V_{f,\p}^\dagger$ be the self-dual twist of the Galois representation attached to $f$ and $\p$ by Deligne, where $F_\p$ is the completion of $F$ at $\p$. We fix a $G_\Q$-stable lattice $T_{f,\p}^\dagger\subset V_{f,\p}^\dagger$ and set $A_{f,\p}^\dagger\defeq V_{f,\p}^\dagger/T_{f,\p}^\dagger$; if the residual representation $A_{f,\p}^\dagger[\p]$ thus obtained is irreducible, then its isomorphism class is independent of the choice of $T_{f,\p}^\dagger$, since all Galois-stable lattices inside $V_{f,\p}^\dagger$ are homothetic. Fix an imaginary quadratic field $K$ of discriminant prime to $Np$ and consider a factorization $N=N^+N^-$ such that a prime number $\ell$ divides $N^+$ (respectively, $N^-$) if and only if $\ell$ splits (respectively, is inert) in $K$. We assume throughout in this section that $N$ is square free and $N^-$ is the product of an even number of primes. 

Let $\cB$ be the (indefinite) quaternion algebra of discriminant $N^-$ and let $\mathcal{R}$ be an Eichler order of level $N^+$ of $\cB$. From now on, with notation as in \S \ref{secJL}, let us set $\phi\defeq\JL(f)$ and take $\wp=\p$, so that $\phi\in\mathcal{S}_k(N^+,N^-;\cO_\p)$. Thus, there are canonical Galois-equivariant isomorphisms 
\begin{equation} \label{galois-isom-eq}
T_{f,\p}^\dagger\simeq T_{\phi,\p}^\dagger,\quad T_{f,\p}^\dagger/\p T_{f,\p}^\dagger\simeq \overline{T}_{\phi,\p}^\dagger,\quad A_{f,\p}^\dagger\simeq A_{\phi,\p}^\dagger.
\end{equation}
Suppose throughout that Assumption \ref{ass} for $\phi=\phi_f$ is satisfied.  


\subsection{Kolyvagin systems}\label{sec-kolsys}

Our goal now is to introduce the Kolyvagin systems built out of Heegner cycles in terms of which we will formulate a higher weight analogue of Kolyvagin's conjecture.

\subsubsection{Kolyvagin primes and Kolyvagin integers}

As above, let $S\in\mathcal{P}=\mathcal{P}(\phi,\p,K)$. 

\begin{definition}\label{defKol}
A \emph{Kolyvagin prime for $(f,S, \p, K)$} is a prime number $\ell$ such that 
\begin{enumerate}
\item $\ell\nmid NSp$;
\item $\ell$ is inert in $K$;
\item $M(\ell)\defeq\min\bigl\{\ord_\p(\ell+1),\ord_\p\bigl(a_\ell(f)\bigr)\bigr\}>0$.  
\end{enumerate} \end{definition}
Let $\mathcal P_{\Kol}^{(S)}(f,\p,K)$ be the set of Kolyvagin primes introduced in Definition \ref{defKol}. Let us write 
\begin{equation}\label{IKol}
\mathcal I_{\Kol}^{(S)}(f,\p,K)\defeq\Bigl\{\text{square-free products of primes in $\mathcal P_{\Kol}^{(S)}(f,\p,K)$}\Bigr\} \end{equation} 
for the set of \emph{Kolyvagin integers} for $(f,S,\p,K)$. 

\begin{remark} \label{kol-rem}
It follows trivially from the definitions that if $S,S'\in\mathcal P$ and $S\,|\,S'$, then there is an inclusion $\mathcal P_{\Kol}^{(S')}(f,\p,K)\subset\mathcal P_{\Kol}^{(S)}(f,\p,K)$, whence $\mathcal I_{\Kol}^{(S')}(f,\p,K)\subset\mathcal I_{\Kol}^{(S)}(f,\p,K)$.
\end{remark}

The data $\p$ and $K$ being regarded as fixed, from now on we shall often write $\mathcal P_{\Kol}^{(S)}(f)$ and $\mathcal I_{\Kol}^{(S)}(f)$ for the two sets above. Notice that $1\in\mathcal I_{\Kol}^{(S)}(f)$. For every $n\in\mathcal I_{\Kol}^{(S)}(f)$, the \emph{Kolyvagin index} of $n$ is
\begin{equation} \label{M(n)-eq}
M(n)\defeq \begin{cases} \min\bigl\{M(\ell)\mid\ell\,|\,n\bigr\}&\text{if $n\geq2$},\\[3mm]\infty&\text{if $n=1$}. \end{cases}
\end{equation}
A Kolyvagin prime for $(f,1,\p,K)$ will be simply called a \emph{Kolyvagin prime for $(f,K)$}; moreover, we set $\mathcal P_{\Kol}(f)\defeq\mathcal P_{\Kol}^{(1)}(f)$ and $\mathcal I_{\Kol}(f)\defeq\mathcal I_{\Kol}^{(1)}(f)$. 
%
%
%
%

\subsubsection{Kolyvagin systems of Heegner cycles}

Following a recipe originally due (in the context of modular abelian varieties) to Kolyvagin (see, \emph{e.g.}, \cite[\S 4]{Gross}, \cite{Kol-Euler}, \cite[\S 3.7]{zhang-selmer}), we attach to $f$ a systematic supply of Galois cohomology classes, which we call \emph{Kolyvagin classes}, that are defined in terms of the Heegner cycles of \S \ref{cycles-subsec}, are indexed over Kolyvagin integers and take values in quotients of $T_{f,\p}^\dagger$. 

To begin with, for all $n\in\mathcal I^{(S)}_{\Kol}(f)$ set 
\[ G_n\defeq \Gal(H_n/H_K),\quad\mathcal G_n\defeq \Gal(H_n/K)\simeq\Pic(\cO_n), \]
so that, by class field theory, $G_n=\prod_{\ell|n}G_\ell$ with $G_\ell$ cyclic of order $\ell+1$, where $\ell$ varies over the prime divisors of $n$. For all $\ell\in\mathcal P^{(S)}_{\Kol}(f)$, we choose a generator $\sigma_\ell$ of $G_\ell$; the Kolyvagin derivative operators are
\[ D_\ell\defeq \sum_{i=1}^\ell i\sigma_\ell^i\in\Z[G_\ell],\quad D_n\defeq \prod_{\ell|n}D_\ell\in\Z[G_n]. \]
In particular, $D_1$ is the identity. 

For a number field $L$, we also define 
\[ \Lambda_{\phi,\p}(L)\defeq\im(\AJ_{\phi,\p,L})\subset H^1\bigl(L,T_{f,\p}^\dagger\bigr), \] 
where $\AJ_{\phi,\p,L}$ is the map in \eqref{AJ1} and the inclusion on the right is a consequence of the first isomorphism in \eqref{galois-isom-eq}. 
As in \cite{LV} (see, in particular, \cite[Corollary 2.7, (3)]{LV} and \cite[Proposition 2.8]{LV}), one can check that for each integer $m$ with $1\leq m\leq n$ there is a natural Galois-equivariant injection
\[ \iota_{\phi,L,m}:\Lambda_{\phi,\p}(L)\big/\p^m\Lambda_{\phi,\p}(L)\longmono H^1\bigl(L,T_{f,\p}^\dagger/\p^m T_{f,\p}^\dagger\bigr), \]
which should the thought of as a higher weight counterpart of the usual Kummer map in the Galois cohomology of abelian varieties over number fields. Analogously, with self-explaining notation, one has a Galois-equivariant injection
\[ \iota_{\phi_S,L,m}:\Lambda_{\phi_S,\p_S}(L)\big/\p_S^m\Lambda_{\phi_S,\p_S}(L)\longmono H^1\bigl(L,T_{\phi_S,\p_S}^\dagger/\p_S^m T_{\phi_S,\p_S}^\dagger\bigr)\simeq H^1\bigl(L,T_{f,\p}^\dagger/\p^m T_{f,\p}^\dagger\bigr), \]
where $\phi_S\in\mathcal S_k(M,N^-S;\cO_{S,\wp_S})$ comes from part (2) of Theorem \ref{levelraising} and the isomorphism on the right follows from the Galois-equivariant isomorphism $T_{\phi_S,\p_S}^\dagger\simeq T_{f,\p}^\dagger$. Of course, there is an equality $\iota_{\phi_1,L,m}=\iota_{\phi,L,m}$.

Recall the general notation introduced in \S \ref{HC}; with $M=N^+$ and $D=N^-$, let us define 
\[ y_{n}(S)\defeq y_{n,N^+,N^-S} (\phi_S)\in\Lambda_{\phi_S,\p_S}(H_n). \] 
Fix $n\in\mathcal I^{(S)}_{\Kol}(f)$, let $\mathcal G$ be a system of representatives for $\mathcal G_n/G_n$ and set
\[ z_{n}(S)\defeq \sum_{\sigma\in\mathcal G}\sigma\Bigl(D_n\bigl(y_{n}(S)\bigr)\!\Bigr)\in\Lambda_{\phi_S,\p_S}(H_n). \]
Then for $S=1$, $z_n=z_n(1)\in \Lambda_{f,\p}(H_n)$. The field extension $H_n/\Q$, which is generalised dihedral, is solvable, so $H^0\bigl(H_n,T_{f,\p}^\dagger/\p^m T_{f,\p}^\dagger\bigr)=0$ by \cite[Lemma 3.10, (2)]{LV}. It follows that restriction induces an isomorphism
\begin{equation} \label{res-iso-eq} 
\res_{H_n/K}:H^1\bigl(K,T_{f,\p}^\dagger/\p^m T_{f,\p}^\dagger\bigr)\overset\simeq\longrightarrow H^1\bigl(H_n,T_{f,\p}^\dagger/\p^m T_{f,\p}^\dagger\bigr)^{\mathcal G_n}. 
\end{equation}
Moreover, one can easily check that if $d_{n,m}(S)$ denotes the image via $\iota_{\phi_S,L,m}$ of the class of $z_{n}(S)$ modulo $\p_S^m$ and $M(n)$ is the Kolyvagin index of $n$ from \eqref{M(n)-eq}, then $d_{n,m}(S)\in H^1\bigl(H_n,T_{f,\p}^\dagger/\p^m T_{f,\p}^\dagger\bigr)^{\mathcal G_n}$ for all $m\in\bigl\{1,\dots,M(n)\bigr\}$. Keeping isomorphism \eqref{res-iso-eq} in mind, for all $m\in\bigl\{1,\dots,M(n)\bigr\}$ we can define the \emph{Kolyvagin class} $c_{n,m}(S)$ as
\[ c_{n,m}(S)\defeq \res_{H_n/K}^{-1}\bigl(d_{n,m}(S)\bigr)\in H^1\bigl(K,T_{f,\p}^\dagger/\p^m T_{f,\p}^\dagger\bigr). \]
Now we can introduce our Kolyvagin systems.
\begin{definition} \label{defKolsys}
The \emph{Kolyvagin system} attached to $(f,\p,K,S)$ is 
\[ \kappa_S\defeq\Bigl\{c_{n,m}(S)\;\Big|\; n\in\mathcal I_{\Kol}^{(S)}(f),\;1\leq m\leq M(n)\Bigr\}. \]
\end{definition}
To simplify notation, for all $n\in\mathcal I_{\Kol}^{(S)}(f)$ set $c_n(S)\defeq c_{n,1}(S)$.
\begin{definition} \label{defKolsys-2}
The \emph{strict Kolyvagin system} attached to $(f,\p,K,S)$ is 
\[ \kappa_S^\star\defeq\Bigl\{c_n(S)\;\Big|\; n\in\mathcal I_{\Kol}^{(S)}(f)\Bigr\}. \]
\end{definition}  
Since $f$, $\p$, $K$ are understood as fixed, our notation does not reflect the dependence of $\kappa_S$ and $\kappa_S^\star$ on these data. By definition, $\kappa_S^\star\subset\kappa_S$.

\subsection{Kolyvagin's conjecture} 

The following conjecture was first proposed, for $S=N^-=1$ and with a slightly different formalism, in \cite[Conjecture A]{Masoero}.

\begin{conjecture}[Kolyvagin's conjecture, higher weight] \label{kolyvagin-conj}
For each $S\in\mathcal{P}^\mathrm{indef}$, $\kappa_S\neq\{0\}$.
\end{conjecture}

This is a higher (even) weight counterpart of a conjecture for rational elliptic curves due to Kolyvagin (\cite[Conjecture A]{kolyvagin-selmer}). Actually, we are interested in the stronger 

\begin{conjecture}[Kolyvagin's conjecture, higher weight, strong form] \label{strong-kolyvagin-conj}
For each $S\in\mathcal{P}^\mathrm{indef}$, $\kappa_S^\star\not=\{0\}$.
\end{conjecture}

The proof (under mild technical assumptions) of Conjecture \ref{strong-kolyvagin-conj} is the main result of this paper (see Theorem \ref{kol-thm}).

\subsection{Triangulation of Selmer groups} \label{triangulation-subsec}

Notation from \S\ref{sec-kolsys} is in force. Let $S\in\mathcal{P}^\mathrm{indef}$. First of all, we need a few auxiliary results that can be found in \cite{E-dVP} and \cite{Masoero} (strictly speaking, in \cite{Masoero} it is assumed that $N^-=S=1$, but the arguments carry over to our more general setting). 

\begin{lemma} \label{MacCallum0} 
Let $x_1,\dots,x_r$ be linearly independent elements of $H^1\bigl(K,\overline{T}_{\phi_f}^\dagger\bigr)$. There exist infinitely many $\ell\in\mathcal{P}_{\Kol}^{(S)}(f)$ such that $\loc_\ell(x_i)\neq0$ for $i=1,\dots,r$.
\end{lemma}

\begin{proof} As in \cite{Masoero}, this is an application of \v{C}ebotarev's density theorem. More precisely, with notation as in \cite{Masoero}, we apply \cite[Lemma 7.5]{Masoero} to the Galois representation $\overline{T}_{\phi_S}^\dagger\simeq \overline{T}_{\phi_f}^\dagger$, the finite subgroup $C$ of $H^1(K,\overline{T}_{\phi_f}^\dagger)$ generated by $\{x_1,\dots,x_r\}$ and a map $\varphi\in\Hom\bigl(C,\overline{T}_{\phi_f}^\dagger\bigr)^+$ such that
$\varphi(x_i)\neq 0$ for $i=1,\dots,r$.
Let $L\defeq K\bigl(\overline{T}_{\phi_f}^\dagger\bigr)$ be the composite of $K$ with the subfield of $\bar\Q$ fixed by the subgroup of $G_\Q$ that acts trivially on $\overline{T}_{\phi_f}^\dagger$. Then \v{C}ebotarev's density theorem implies that there exist infinitely many $\ell\in\mathcal{P}_{\Kol}^{(S)}(f)$ such that $\varphi$ corresponds to an arithmetic Frobenius $\Frob_\lambda$ under the natural isomorphism
\[ \Gal(L_C/L)\overset\simeq\longrightarrow\Hom\bigl(C,\overline{T}_{\phi_f}^\dagger\bigr),\quad\sigma\longmapsto\varphi_\sigma \]
for a suitable finite Galois extension $L_C/L$ and some prime $\lambda$ of $L$ above $\ell$. Avoiding the finitely many primes at which the classes $x_1,\dots, x_r$ ramify, we find infinitely many $\ell\in\mathcal{P}_{\Kol}^{(S)}(f)$ with the previous property such that $\Frob_\lambda$ generates the decomposition group of a prime of $L_C$ above $\lambda$. Finally, the lemma follows by observing that $\loc_\ell(x_i)=0$ if and only if $\varphi(x_i)=\varphi_{\Frob_\lambda}(x_i)=0$. \end{proof}

Let $k_\p\defeq\cO_\p/\p\cO_\p$ be the residue field of $\cO_F$ at $\p$. In what follows, the superscript $\pm$ will denote the eigenspace on which the non-trivial element of $\Gal(K_\ell/\Q_\ell)\simeq\Gal(K/\Q)$ acts as multiplication by $\pm1$.

\begin{lemma} \label{dimension 1}
Let $\ell\in\mathcal{P}_{\Kol}^{(S)}(f)$ and let $\bullet\in\{\fin,\sing\}$. The $k_\p$-vector space $H^1_\bullet\bigl(K_\ell,\overline{T}_{f,\p}^\dagger\bigr)^\pm$ is $1$-dimensional. 
\end{lemma}

\begin{proof} See, \emph{e.g.}, \cite[\S 5.3]{E-dVP}. \end{proof}

\begin{lemma}\label{MacCallum}
Let $T\in\mathcal{I}_{\Kol}^{(S)}(f)$ and $\ell\in\mathcal{P}_{\Kol}^{(S)}(f)$ with $\ell\nmid T$. Each group $\Sel_{S,(\ell),[T]}\bigl(K,\overline{T}_{f,\p}^\dagger\bigr)^\pm$ is non-trivial. 
\end{lemma}

\begin{proof} See, \emph{e.g.}, the proof of \cite[Proposition 7.8]{Masoero}. \end{proof}

\begin{lemma} \label{McCallum1}
Let $\ell\in\mathcal{P}_{\Kol}^{(S)}(f)$ and $n\in\mathcal{I}_{\Kol}^{(S)}(f)$ with $\ell\nmid n$. There is an isomorphism 
\[ \vartheta_\ell:H^1_\mathrm{fin}\bigl(K_\ell,\overline{T}_{f,\p}^\dagger\bigr)\longrightarrow H^1_\mathrm{sing}\bigl(K_\ell,\overline{T}_{f,\p}^\dagger\bigr) \] 
and an equality 
\[ \loc_\ell\bigl(c_{n\ell}(S)\bigr)=\vartheta_\ell\Bigl(\loc_\ell\bigl(c_{n}(S)\bigr)\!\Bigr). \] 
\end{lemma}

\begin{proof} See, \emph{e.g.}, \cite[Proposition 5.7, (iv)]{E-dVP}. \end{proof}

The next definition introduces a notion that will prove useful in subsequent arguments.

\begin{definition} 
A \emph{base point} of $\kappa_S^\star$ is a prime number $\ell$ such that 
\begin{itemize}
\item $\ell\nmid D_KNp$; 
\item $\loc_\ell\bigl(c_n(S)\bigr)=0$ for all $n\in\mathcal{I}_{\Kol}^{(S)}(f)$.
\end{itemize}
The \emph{base locus} of $\kappa_S^\star$ is the set $\cB_S^\star$ of all base points of $\kappa_S^\star$.
\end{definition}

For every $n\in\mathcal I_{\Kol}^{(S)}(f)$, let us write $\nu(n)$ for the number of prime factors of $n$; recall the Kolyvagin index $M(n)$ of $n$ that was introduced in \eqref{M(n)-eq}.

\begin{definition} \label{vanishing-def}
Assume that $\kappa_S\neq\{0\}$. The \emph{vanishing order} of $\kappa_S$ is
\[ \nu_S\defeq\min\Bigl\{\nu(n)\;\Big|\;\text{$n\in\mathcal I_{\Kol}^{(S)}(f)$ and $c_{n,m}(S)\not=0$ for some $m\leq M(n)$}\Bigr\}\in\N. \]
\end{definition}

We can give a similar definition for $\kappa_S^\star$, as follows.

\begin{definition} \label{vanishing-def-2}
Assume that $\kappa_S^\star\neq\{0\}$. The \emph{vanishing order} of $\kappa_S^\star$ is
\[ \nu_S^\star\defeq\min\Bigl\{\nu(n)\;\Big|\;\text{$n\in\mathcal I_{\Kol}^{(S)}(f)$ and $c_n(S)\not=0$}\Bigr\}\in\N. \]
\end{definition}
Clearly, $\nu_S\leq\nu_S^\star$. Let $\epsilon(f)\in\{\pm1\}$ be the root number of $f$ over $\Q$, \emph{i.e.}, the sign of the functional equation of $L(f,s)$, and define 
\[ \epsilon_S\defeq\epsilon(f)\cdot(-1)^{\nu_S^\star+1}\in\{\pm1\}. \] 
The next result, which is implicit in \cite{Masoero}, is key for proving Kolyvagin's conjecture, which will be done in \S\ref{sec KolConj}. Our proof follows the arguments in \cite[Lemma 8.4]{zhang-selmer}; for completeness, we work out the strategy from \cite{zhang-selmer} in our more general setting. 

\begin{proposition}[Triangulation of Selmer groups]\label{prop masoero}
Assume that $\kappa_S^\star\neq\{0\}$. Then 
\begin{enumerate}
\item $\dim_{k_\p}\Bigl(\Sel_{S}\bigl(K,\overline{T}_{f,\p}^\dagger\bigr)^{\epsilon_S}\Bigr)=\nu_S^\star+1$;
\item $\dim_{k_\p}\Bigl(\Sel_{S}\bigl(K/\overline{T}_{f,\p}^\dagger\bigr)^{-\epsilon_S}\Bigr)\leq\nu_S^\star$; 
\item $\Sel_{S,(\cB_S^\star)}\bigl(K,\overline{T}_{f,\p}^\dagger\bigr)^{\epsilon_S}=\Sel_{S}\bigl(K,\overline{T}_{f,\p}^\dagger\bigr)^{\epsilon_S}$;
\item $\dim_{k_\p}\Bigl(\Sel_{S,(\cB_S^\star)}\bigl(K,\overline{T}_{f,\p}^\dagger\bigr)^{-\epsilon_S}\!\Bigr)\leq\nu_S^\star$. 
\end{enumerate}
\end{proposition} 

\begin{proof} First of all, recall that if $n\in\mathcal{I}_{\Kol}^{(S)}(f)$, then $c_n(S)\in H^1\bigl(K,\overline{T}_{f,\p}^\dagger\bigr)^{\epsilon(f)\cdot(-1)^{\nu(n)+1}}$ (see, \emph{e.g.}, \cite[Proposition 3.14]{LV}). We divide our proof into four steps.

\texttt{Step $1$}. We claim that for $j\in\{0,\dots,\nu_S^\star\}$ there are prime numbers $\ell_1,\dots,\ell_{\nu_S^\star+j}$ such that 
\begin{itemize} 
\item for all $i=1,\dots,j+1$, we have $c_{n_i}(S)\neq 0$ for $n_i\defeq\ell_i\cdots\ell_{\nu_S^\star+i-1}$;
\item for all $i=1,\dots,j$, we have $\loc_{\ell_{\nu_S^\star+i}}\bigl(c_{n_i}(S)\bigr)\neq 0$.
\end{itemize} 
Suppose $j=0$. Our assumption on $\kappa_S^\star$ ensures that there exists $n_1=\ell_1\cdots\ell_{\nu_S^\star}$ such that $c_{n_1}(S)\neq0$, which shows the existence of primes $\ell_1,\dots,\ell_{\nu_S^\star}$ satisfying the conditions above (the second being void in this case). Suppose we have defined primes $\ell_1,\dots,\ell_{\nu_S^\star+j}$ for some $j\in\{0,\dots,\nu_S^\star-1\}$ satisfying the three conditions above. By Lemma \ref{MacCallum}, we can choose $c\neq0$ in $H^1\bigl(K,\overline{T}_{f,\p}^\dagger\bigr)^{-\epsilon_S}$ such that $c\in\Sel_{S,(\ell_{j+1}),[T_j]}\bigl(K,\overline{T}_{f,\p}^\dagger\bigr)^{-\epsilon_S}$ for $T_j\defeq\ell_{j+2}\cdots\ell_{\nu_S^\star+j}$ if $\nu_S^\star\geq 2$ and $T_j\defeq1$ otherwise. By the observation at the beginning of the proof, the classes $c_{n_{j+1}}(S)$ and $c$ belong to different eigenspaces for the action of $\Gal(K/\Q)$, so they are linearly independent. By Lemma \ref{MacCallum0}, we can find a prime $\ell_{\nu_S^\star+j+1}\in\mathcal{I}_{\Kol}^{(S)}(f)$, distinct from $\ell_1,\dots,\ell_{\nu_S^\star+j}$, such that 
\begin{itemize}
\item $\loc_{\ell_{\nu_S^\star+j+1}}\bigl(c_{n_{j+1}}(S)\bigr)\neq 0$;
\item $\loc_{\ell_{\nu_S^\star+j+1}}(c)\neq 0$. 
\end{itemize}
To complete \texttt{Step $1$}, we show that $c_{n_{j+2}}(S)\neq 0$ for $n_{j+2}\defeq\ell_{j+2}\cdots\ell_{\nu_S^\star+j+1}$. Notice that $c$ and $c_{n_{j+1}\ell_{\nu_S^\star+j+1}}(S)$ belong to the same eigenspace for the action of $\Gal(K/\Q)$ and that these two classes are orthogonal with respect to the local Tate pairing at all places except possibly at $\ell_{j+1}$ and $\ell_{\nu_S^\star+j+1}$. The localisations of both $c$ and $c_{n_{j+1}\ell_{\nu_S^\star+j+1}}(S)$ at  $\ell_{\nu_S^\star+j+1}$ are non-zero, the non-vanishing of the latter class stemming from Lemma \ref{McCallum1}. Thus, the local Tate pairing at $\ell_{\nu_S^\star+j+1}$ is non-zero by Lemma \ref{dimension 1}. It follows from global reciprocity that the local Tate pairing of $c$ and $c_{n_{j+1}\ell_{\nu_S^\star+j+1}}(S)$ is non-zero at $\ell_{j+1}$ as well, so $\loc_{\ell_{j+1}}(c)\neq 0$ and $\loc_{\ell_{j+1}}\bigl(c_{n_{j+1}\ell_{\nu_S^\star+j+1}}(S)\bigr)\neq0$. Then, again by Lemma \ref{McCallum1}, we have $\loc_{\ell_{j+1}}\bigl(c_{n_{j+2}}(S)\bigr)\neq 0$, whence $c_{n_{j+2}}(S)\neq 0$. This completes \texttt{Step $1$}. 

\texttt{Step $2$}. Pick a prime $\ell_{2\nu_S^\star+1}$ such that $\loc_{\ell_{2\nu_S^\star+1}}\bigl(c_{n_{\nu_{S^\star}+1}}(S)\bigr)\neq0$: such a prime exists by Lemma \ref{MacCallum0} because $c_{n_{\nu_S^\star+1}}(S)\neq 0$. Summing up, we have built a sequence of primes $\ell_1,\dots,\ell_{2\nu_S^\star+1}$ satisfying the following properties: 
\begin{itemize} 
\item the classes $c_{n_1}(S),\dots,c_{n_{\nu_S^\star+1}}(S)$ in $H^1\bigl(K,\overline{T}_{f,\p}^\dagger\bigr)$ are linearly independent;
\item for all $i=1,\dots,\nu_S^\star+1$ and all $j=i,\dots,\nu_S^\star+i-1$, we have $\loc_{\ell_j}\bigl(c_{n_i}(S)\bigr)=0$;
\item for all $i=1,\dots,\nu_S^\star+1$, we have $\loc_{\ell_{\nu_S^\star+i}}\bigl(c_{n_i}(S)\bigr)\neq0$,
\end{itemize}
Here the second condition follows from the fact that, by definition of $\nu_S^\star$, for each prime $\ell\,|\,n_i$ the class $c_{n_i/\ell}(S)$ is trivial, and so by Lemma \ref{McCallum1} the localization of $c_{n_i}(S)$ at $\ell$ is zero. In particular, the classes $c_{n_i}(S)$ belong to $\Sel_S\bigl(K,\overline{T}_{f,\p}^\dagger\bigr)^{\epsilon_S}$; moreover, they are linearly independent because, for each $i=1,\dots,\nu_S^\star +1$, we have $\loc_{\ell_{\nu_S^\star+i}}(c_{n_i})\neq 0$ and $\loc_{\ell_{\nu_S^\star+i}}(c_{n_j})=0$ for $j=i+1,\dots,\nu_S^\star+1$. 

\texttt{Step $3$}. We want to show that the classes $c_{n_i}(S)$ for $i=1,\dots,\nu_S^\star+1$ form bases of both $\Sel_S\bigl(K,\overline{T}_{f,\p}^\dagger\bigr)^{\epsilon_S}$ and $\Sel_{S,(\cB_S^\star)}\bigl(K,\overline{T}_{f,\p}^\dagger\bigr)^{\epsilon_S}$, thus showing (1) and (3) simultaneously. There is an inclusion $\Sel_S\bigl(K,\overline{T}_{f,\p}^\dagger\bigr)^{\epsilon_S}\subset\Sel_{S,(\cB_S^\star)}\bigl(K,\overline{T}_{f,\p}^\dagger\bigr)^{\epsilon_S}$, so in order to complete \texttt{Step $3$} it remains to prove that the above-mentioned classes generate $\Sel_{S,(\cB_S^\star)}\bigl(K,\overline{T}_{f,\p}^\dagger\bigr)^{\epsilon_S}$. Fix $c\in\Sel_{S,(\cB_S^\star)}\bigl(K,\overline{T}_{f,\p}^\dagger\bigr)^{\epsilon_S}$. At the cost of subtracting a linear combination of the classes $c_{n_i}(S)$, we can assume that $\loc_{\ell_{\nu_S^\star+j}}(c)=0$ for all $j=1,\dots,\nu_S^\star+1$. Therefore, it is enough to check that any such class is trivial. Let us set $n'\defeq\ell_{\nu_S^\star+1}\cdots\ell_{2\nu_S^\star+1}$. Then $c_{n'}(S)\neq0$ because its localisation at $\ell_{2\nu_S^\star+1}$ is non-zero: this follows from Lemma \ref{McCallum1} and the fact that $\loc_{\ell_{2\nu_S^\star+1}}\bigl(c_{n_{\nu_S^\star+1}}(S)\bigr)\neq0$. Moreover, note that $c_{n'}(S)$ and $c$ belong to different eigenspaces for the action of $\Gal(K/\Q)$. Suppose $c\neq0$ and pick a prime $\ell_{2\nu_S^\star+2}\notin\{\ell_1,\dots,\ell_{2\nu_S^\star+1}\}$ such that $\loc_{\ell_{2\nu_S^\star+2}}(c)\neq 0$ and $\loc_{\ell_{2\nu_S^\star+2}}\bigl(c_{n'}(S)\bigr)\neq 0$, which exists by Lemma \ref{MacCallum0}. Set $n''\defeq n'\ell_{2\nu_S^\star+2}$. By Lemma \ref{McCallum1}, we have $c_{n''}(S)\neq 0$ because $\loc_{\ell_{2\nu_S^\star+2}}\bigl(c_{n'}(S)\bigr)\neq0$. Furthermore, $c_{n''}(S)$ and $c$ lie in the same eigenspace for the action of $\Gal(K/\Q)$. Let us write the Tate pairing as 
\begin{equation} \label{tate-eq}
0=\sum_{v\in\cB_S^\star}\big\langle  c,c_{n''}(S)\big\rangle_v+\sum_{\ell\mid n''}\big\langle c,c_{n''}(S)\big\rangle_\ell. 
\end{equation}
By definition of the base locus, the first sum in \eqref{tate-eq} is zero. Furthermore, all local pairings at the primes $\ell_{\nu_S^\star+1},\dots,\ell_{2\nu_S^\star+1}$ are zero, so the right-hand sum in \eqref{tate-eq} reduces to the term $\big\langle c,c_{n''}(S)\big\rangle_{\ell_{2\nu_S^\star+2}}$, which is non-zero by Lemma \ref{dimension 1}, as the localisations of both classes are non-zero. This contradiction gives the vanishing of $c$, as was to be shown. 

\texttt{Step 4}. In light of the inclusion $\Sel_{S}\bigl(K/\overline{T}_{f,\p}^\dagger\bigr)^{-\epsilon_S}\subset\Sel_{S,(\cB_S^\star)}\bigl(K/\overline{T}_{f,\p}^\dagger\bigr)^{-\epsilon_S}$, in order to show (2) and (4) it suffices to show (4). Let us choose primes $\ell_{1},\dots,\ell_{2\nu_S^\star}$ as in \texttt{Step 1}. By contradiction, assume the inequality $\dim_{k_\p}\Bigl(\Sel_{S,(\cB_S^\star)}\bigl(K/\overline{T}_{f,\p}^\dagger\bigr)^{-\epsilon_S}\Bigr)\geq\nu_S^\star+1$. Pick $\nu_S^\star+1$ linearly independent elements $d_1,\dots,d_{\nu_S^\star+1}$ of $\Sel_{S,(\cB_S^\star)}\bigl(K/\overline{T}_{f,\p}^\dagger\bigr)^{-\epsilon_S}$ and for each $i\in\{1,\dots,\nu_S^\star+1\}$ consider the $\nu_S^\star$-tuple $\bigl(\loc_{\ell_{\nu_S^\star+1}}(d_i),\dots,\loc_{\ell_{2\nu_S^\star}}(d_i)\bigr)$. In light of Lemma \ref{dimension 1}, in this way we get $\nu_S^\star+1$ vectors in a $k_\p$-vector space of dimension $\nu_S^\star$, so there is a linear combination $d$ of them satisfying $\loc_{\ell_{\nu_S^\star+i}}(d)=0$ for all $i=1,\dots,\nu_S^\star$. It is enough to show that $d=0$. Suppose $d\neq 0$. The classes $d$ and $c_{n_{\nu_S^\star+1}}$ belong to different eigenspaces, so we can choose, by Lemma \ref{MacCallum0}, a prime $\ell_{2\nu_S^\star+1}$ (possibly different from the prime with the same name that was chosen in \texttt{Step 2}) such that $\loc_{\ell_{2\nu_S^\star+1}}(d)\neq 0$ and $\loc_{\ell_{2\nu_S^\star+1}}(c_{n_{\nu_S^\star+1}})\neq0$. Now set $n'\defeq\ell_{\nu_S^\star+1}\cdots\ell_{2\nu_S^\star+1}$. As above, $c_{n'}(S)\neq0$ because, by Lemma \ref{McCallum1} and the fact that $\loc_{\ell_{2\nu_S^\star+1}}\bigl(c_{n_{\nu_S^\star+1}}(S)\bigr)\neq 0$, its localisation at $\ell_{2\nu_S^\star+1}$ is non-zero. Moreover, as above, the Tate pairing reduces to $\big\langle d,c_{n'}(S)\big\rangle_{\ell_{2\nu_S^\star+1}}=0$; while, on the other hand, $\big\langle d,c_{n'}(S)\big\rangle_{\ell_{2\nu_S^\star+1}}\neq0$ by Lemma \ref{dimension 1}, as the localisations of both classes are non-zero. This contradiction proves the vanishing of $d$, as desired. \end{proof}

\subsection{Proof of Kolyvagin's conjecture} \label{sec KolConj}

Our goal now is to prove Kolyvagin's conjecture in its strong form (Conjecture \ref{strong-kolyvagin-conj}). As we shall see, in our arguments we use in a crucial way the explicit reciprocity laws from \S \ref{reciprocity-subsec} together with results by Skinner--Urban (\cite{SU}). 
 
We begin with a special case of the parity conjecture for Selmer groups. 

\begin{theorem}[Parity conjecture, odd case] \label{parity}
If $S\in\mathcal{P}^\mathrm{indef}$, then $\dim_{k_\p}\Bigl(\Sel_S\bigl(K,\overline{T}_{f,\p}^\dagger\bigr)\!\Bigr)$ is odd. 
\end{theorem}

\begin{proof} First of all, we observe that if $T\in\mathcal{P}^\mathrm{indef}$, then $\Sel_T\bigl(K,\overline{T}_{f,\p}^\dagger\bigr)\neq 0$. To show this, we consider the Galois representation $V_{f_T,\p_T}^\dagger$ attached to the eigenform $f_T$ from Notation \ref{JL-notation}; let $T_{f_T,\p_T}^\dagger\subset V_{f_T,\p_T}^\dagger$ be a Galois-stable lattice and let $A_{f_T,\p_T}^\dagger\defeq V_{f_T,\p_T}^\dagger\big/T_{f_T,\p_T}^\dagger$, as before. The functional equation satisfied by the complex $L$-function $L(f_T/K,s)$ gives $L(f_T/K,1)=0$. A result due to Skinner--Urban (\cite{SU}; \emph{cf.} also \cite[Theorem 4.6]{wang} for a more precise statement) shows that $\Sel\bigl(K,A_{f_T,\p_T}\bigr)\neq0$, hence $\Sel_T\bigl(K,\overline{T}_{f,\p}^\dagger\bigr)\neq 0$ by Lemma \ref{lemma selmer}. 

Arguing by contradiction, let us suppose now that $\dim_{k_\p}\Bigl(\Sel_S\bigl(K,\overline{T}_{f,\p}^\dagger\bigr)\!\Bigr)$ is even. By the observation above applied to $T=S$, we may assume that $r\defeq\dim_{k_\p}\Bigl(\Sel_S\bigl(K,\overline{T}_{f,\p}^\dagger\bigr)\!\Bigr)\geq2$. Thanks to \v{C}ebotarev's density theorem, we can choose $x_1\in\Sel_S\bigl(K,\overline{T}_{f,\p}^\dagger\bigr)\smallsetminus\{0\}$ and an admissible prime $\ell_1$ as in Theorem \ref{levelraising} relative to $f_S$ such that $\loc_{\ell_1}(x_1)\neq 0$. Therefore, localisation induces a surjection $\loc_{\ell_1}:\Sel_{S}\bigl(K,\overline{T}_{f,\p}^\dagger\bigr)\twoheadrightarrow H^1_{\mathrm{fin}}\bigl(K_{\ell_1},\overline{T}_{f,\p}^\dagger\bigr)$ and, by Lemma \ref{lemma7.2}, there is an equality 
\[ \dim_{k_\p}\Bigl(\Sel_{{S\ell_1}}\bigl(K,\overline{T}_{f,\p}^\dagger\bigr)\!\Bigr)= r-1\geq 1. \]
As before, by \v{C}ebotarev's density theorem we can choose $x_2\in \Sel_{{S\ell_1}}\bigl(K,\overline{T}_{f,\p}^\dagger\bigr)\smallsetminus\{0\}$ and an admissible prime $\ell_2$ such that $\loc_{\ell_2}(x_2)\neq0$. Thus, $\loc_{\ell_2}:\Sel_{{S\ell_1}}\bigl(K,\overline{T}_{f,\p}^\dagger\bigr)\twoheadrightarrow H^1_\mathrm{fin}\bigl(K_{\ell_2},\overline{T}_{f,\p}^\dagger\bigr)$ is surjective and, by Lemma \ref{lemma7.2}, one has
\[ \dim_{k_\p}\Bigl(\Sel_{{S\ell_1\ell_2}}\bigl(K,\overline{T}_{f,\p}^\dagger\bigr)\!\Bigr)=\dim_{k_\p}\Bigl(\Sel_{{S\ell_1}}\bigl(K,\overline{T}_{f,\p}^\dagger\bigr)\!\Bigr)-1= r-2. \] 
Repeating this process, we find $T=S\ell_1\cdots\ell_{2r}$ such that $\dim_{k_\p}\Bigl(\Sel_{{T}}\bigl(K,\overline{T}_{f,\p}^\dagger\bigr)\!\Bigr)=0$, which contradicts the observation at the beginning of the proof. \end{proof}

%
The next result shows that Kolyvagin's conjecture holds true when the Selmer rank is $1$.

\begin{theorem}[Wang] \label{rank 0}
Let $S\in\mathcal{P}^\mathrm{indef}$ and assume that $\dim_{k_\p}\Bigl(\Sel_{S}\bigl(K,\overline{T}_{f,\p}^\dagger\bigr)\!\Bigr)=1$. The class $c_1(S)\in H^1\bigl(K,\overline{T}_{f,\p}^\dagger\bigr)$ is not zero; in particular, $\kappa_S^\star\neq\{0\}$.   
\end{theorem}

\begin{proof}  This is \cite[Theorem 2]{wang}; for the reader's convenience, we review the proof from \cite{wang}. 

\texttt{Step 1}. Pick $x\in\Sel_{S}\bigl(K,\overline{T}_{f,\p}^\dagger\bigr)\smallsetminus\{0\}$, so that $\{x\}$ is a basis of $\Sel_{S}\bigl(K,\overline{T}_{f,\p}^\dagger\bigr)$ over $k_\p$. By the \v{C}ebotarev density theorem, we can choose an admissible prime $\ell$ relative to $(f,K)$ such that $\loc_\ell(x)\neq 0$. It follows that localisation induces a surjection 
\[ \loc_\ell:\Sel_{S}\bigl(K,\overline{T}_{f,\p}^\dagger\bigr)\longepi H^1_\fin\bigl(K_\ell,\overline{T}_{f,\p}^\dagger\bigr), \] 
and then we deduce from Lemma \ref{lemma7.2} that $\Sel_{{S\ell}}\bigl(K,\overline{T}_{f,\p}^\dagger\bigr)=0$. 

\texttt{Step 2}. We show that $\lambda_f(S\ell)\defeq\lambda_{f,1}(S\ell)$ is a $p$-adic unit. As before, fix a Galois-stable lattice $T_{f_{S\ell},\p_{S\ell}}^\dagger$ inside $V_{f_{S\ell},\p_{S\ell}}^\dagger$ and set $A_{f_{S\ell},\p_{S\ell}}^\dagger\defeq V_{f_{S\ell},\p_{S\ell}}^\dagger\big/T_{f_{S\ell},\p_{S\ell}}^\dagger$. Let $L^\mathrm{alg}({f}_{S\ell}/K)$ be the algebraic part of the special 
value of $L(f_{S\ell}/K,s)$ at $s=1$, normalised as in \cite{wang} using Hida's canonical period attached to $f_{S\ell}$. By \cite[Theorem 4.5]{wang}, we know that 
\[ \ord_{\p_{S\ell}}\Bigl(L^\mathrm{alg}(f_{S\ell}/K)\!\Bigr)=\length_{\cO_{S\ell}}\Bigl(\Sel\bigl(K,A^\dagger_{{f}_{S\ell},\p_{S\ell}}\bigr)\!\Bigr)+\sum_{q\mid NS\ell}t_{{f}_{S\ell}}(q), \]
where $t_{{f}_{S\ell}}(q)$ is the Tamagawa exponent of $A_{{f}_{S\ell},\wp_{S\ell}}^\dagger$ at $q$.
Since $A_{f_{S\ell},\p_{S\ell}}^\dagger[\p_{S\ell}]$ is an irreducible $G_K$-module, there is an isomorphism $\Sel_{{S\ell}}\bigl(K,A_{f,\p}^\dagger[\p_{S\ell}]\bigr)\simeq\Sel_{{S\ell}}\bigl(K,A_{f,\p}^\dagger\bigr)[\p_{S\ell}]$; therefore, we conclude from Lemma \ref{lemma selmer} and \texttt{Step 1} that 
\begin{equation} \label{ord-eq1}
\ord_{\p_{S\ell}}\Bigl(L^\mathrm{alg}(f_{S\ell}/K)\!\Bigr)=\sum_{q\mid NS\ell}t_{{f}_{S\ell}}(q). 
\end{equation} 
Thanks to a formula by Chida--Hsieh, generalised by Wang in \cite[Theorem 3.1]{wang}, there is an equality
\begin{equation} \label{ord-eq2}
\ord_{\p_{S\ell}}\Bigl(L^\mathrm{alg}({{f}_{S\ell}}/K)\!\Bigr)=2\cdot\ord_{\p_{S\ell}}\bigl(\theta_{1,N^+,N^-S\ell}({f_{S\ell}})\bigr)+\sum_{q\mid NS\ell}t_{{f}_{S\ell}}(q). 
\end{equation}
Combining \eqref{ord-eq1} and \eqref{ord-eq2}, we get $\lambda_f(S\ell)=\theta_{1,N^+,N^-S\ell}({f_{S\ell}})\in\cO_{S\ell,\p_{S\ell}}^\times$.

\texttt{Step 3}. Since, by the previous step, $\lambda_f(S\ell)$ is a $p$-adic unit, Theorem \ref{ERL2} implies that $\loc_\ell\bigl(c_1(S)\bigr)\neq0$, whence $c_1(S)\not=0$. \end{proof}


Now we can prove the main result of this paper (Theorem A in the introduction).

\begin{theorem}[Kolyvagin's conjecture, strong form] \label{kol-thm}
If $S\in\mathcal{P}^\mathrm{indef}$, then $\kappa_S^\star\neq\{0\}$.  
\end{theorem}

\begin{proof} As in \cite{zhang-selmer}, we proceed by induction on $r\defeq\dim_{k_\p}\Bigl(\Sel_{S}\bigl(K,\overline{T}_{f,\p}^\dagger\bigr)\!\Bigr)$. Thanks to Theorem \ref{parity}, we know that $r$ is odd. Pick $\epsilon\in\{+,-\}$ such that 
\begin{equation} \label{eq7.1}
\dim_{k_\p}\Bigl(\Sel_{S}\bigl(K,\overline{T}_{f,\p}^\dagger\bigr)^\epsilon\Bigr)>\dim_{k_\p }\Bigl(\Sel_{S}\bigl(K,\overline{T}_{f,\p}^\dagger\bigr)^{-\epsilon}\Bigr).
\end{equation}
When $r=1$, Theorem \ref{rank 0} ensures that $\kappa_S^\star\neq\{0\}$, so we may assume $r\geq3$ and that $\kappa_U^\star\neq\{0\}$ for every $U\in\mathcal P^\mathrm{indef}$ such that $\dim_{k_\p}\Bigl(\Sel_U\bigl(K,\overline{T}_{f,\p}^\dagger\bigr)\!\Bigr)=r-2$. The first part of the proof is similar to the proof of Theorem \ref{rank 0}. Choose $x_1\in\Sel_{S}\bigl(K,\overline{T}_{f,\p}^\dagger\bigr)^\epsilon\smallsetminus\{0\}$ and an admissible prime $\ell_1$ relative to $f_S$ such that $\loc_{\ell_1}(x_1)\neq 0$; then $\loc_{\ell_1}:\Sel_{S}\bigl(K,\overline{T}_{f,\p}^\dagger\bigr)^\epsilon\rightarrow H^1_\fin\bigl(K_{\ell_1},\overline{T}_{f,\p}^\dagger\bigr)$ is surjective, and Lemma \ref{lemma7.2} ensures that
\begin{equation}\label{eq-rank1}
\dim_{k_\p }\Bigl(\Sel_{{S\ell_1}}\bigl(K,\overline{T}_{f,\p}^\dagger\bigr)\!\Bigr)= r-1\geq 2.
\end{equation}
Since $r\geq3$, and hence $\dim_{k_\p }\Bigl(\Sel_{S}\bigl(K,\overline{T}_{f,\p}^\dagger\bigr)^\epsilon\Bigr)\geq2$ by \eqref{eq7.1}, Lemma \ref{lemma7.2} also shows that 
\begin{equation}\label{eq7.2}
\dim_{k_\p }\Bigl(\Sel_{{S\ell_1}}\bigl(K,\overline{T}_{f,\p}^\dagger\bigr)^\epsilon\Bigr)=\dim_{k_\p }\Bigl(\Sel_{S}\bigl(K,\overline{T}_{f,\p}^\dagger\bigr)^\epsilon\Bigr)-1\geq1
\end{equation}
and so, using \eqref{eq-rank1}, we get
\begin{equation}\label{eq7.3}
\dim_{k_\p }\Bigl(\Sel_{{S\ell_1}}\bigl(K,\overline{T}_{f,\p}^\dagger\bigr)^{-\epsilon}\Bigr)=\dim_{k_\p }\Bigl(\Sel_{S}\bigl(K,\overline{T}_{f,\p}^\dagger\bigr)^{-\epsilon}\Bigr).
\end{equation}
By \eqref{eq7.2}, we can choose $x_2\in\Sel_{S\ell_1}\bigl(K,\overline{T}_{f,\p}^\dagger\bigr)^\epsilon\smallsetminus\{0\}$ and an admissible prime $\ell_2$ relative to $f_{S\ell_1}$ such that $\loc_{\ell_2}(x_2)\not=0$; then, by Lemma \ref{lemma7.2}, we have
\[ \dim_{k_\p }\Bigl(\Sel_{{S\ell_1\ell_2}}\bigl(K,\overline{T}_{f,\p}^\dagger\bigr)\!\Bigr)=\dim_{k_\p }\Bigl(\Sel_{{S\ell_1}}\bigl(K,\overline{T}_{f,\p}^\dagger\bigr)\!\Bigr)-1= r-2. \] 
Now $S\ell_1\ell_2\in\mathcal{P}^\mathrm{indef}$, so by the inductive hypothesis we conclude that $\kappa_{S\ell_1\ell_2}^\star\neq \{0\}$. By Corollary \ref{coro2.5}, and keeping Remark \ref{kol-rem} in mind, for all $n\in\mathcal I_{\Kol}^{(S\ell_1\ell_2)}(f)$ we have
\[ \loc_{\ell_1}\bigl(c_n(S)\bigr)\neq0\;\Longleftrightarrow\;\loc_{\ell_2}\bigl(c_n(S\ell_1\ell_2)\bigr)\neq0. \] 
Thus, in order to prove that $c_n(S)\neq0$ it is enough to show that $\ell_2$ is not a base point for $\kappa_{S\ell_1\ell_2}^\star$. Suppose, by contradiction, that $\ell_2$ is a base point for $\kappa_{S\ell_1\ell_2}^\star$. There is an inclusion 
\begin{equation} \label{inclusion}
\Sel_{{S\ell_1}}\bigl(K,\overline{T}_{f,\p}^\dagger\bigr)^\pm\subset\Sel_{{S\ell_1\ell_2},(\cB_{S\ell_1\ell_2})}\bigl(K,\overline{T}_{f,\p}^\dagger\bigr)^\pm.
\end{equation}
Since $\dim_{k_\p }\Bigl(\Sel_{{S\ell_1\ell_2}}\bigl(K,\overline{T}_{f,\p}^\dagger\bigr)\!\Bigr)$ is again odd, it follows that 
\[ \dim_{k_\p }\Bigl(\Sel_{{S\ell_1\ell_2}}\bigl(K,\overline{T}_{f,\p}^\dagger\bigr)^\epsilon\Bigr)\not=\dim_{k_\p }\Bigl(\Sel_{{S\ell_1\ell_2}}\bigl(K,\overline{T}_{f,\p}^\dagger\bigr)^{-\epsilon}\Bigr). \]
First suppose $\dim_{k_\p }\Bigl(\Sel_{{S\ell_1\ell_2}}\bigl(K,\overline{T}_{f,\p}^\dagger\bigr)^\epsilon\Bigr)>\dim_{k_\p }\Bigl(\Sel_{{S\ell_1\ell_2}}\bigl(K,\overline{T}_{f,\p}^\dagger\bigr)^{-\epsilon}\Bigr)$. By Proposition \ref{prop masoero} and the non-triviality of $\kappa_{S\ell_1\ell_2}^\star$, we have $\epsilon=\epsilon_{S\ell_1\ell_2}$. Then Proposition \ref{prop masoero} gives 
\[ \Sel_{{S\ell_1\ell_2}}\bigl(K,\overline{T}_{f,\p}^\dagger\bigr)^\epsilon=\Sel_{{S\ell_1\ell_2},(\cB_{S\ell_1\ell_2}^\star)}\bigl(K,\overline{T}_{f,\p}^\dagger\bigr)^\epsilon, \] 
so from \eqref{inclusion} there is also an inclusion 
\[ \Sel_{{S\ell_1}}\bigl(K,\overline{T}_{f,\p}^\dagger\bigr)^\epsilon\subset\Sel_{{S\ell_1\ell_2}}\bigl(K,\overline{T}_{f,\p}^\dagger\bigr)^\epsilon. \]
Finally, $x_2\in \Sel_{{S\ell_1}}\bigl(K,\overline{T}_{f,\p}^\dagger\bigr)^\epsilon$ and $\loc_{\ell_2}(x_2)\neq 0$, which contradicts part (1) of Lemma \ref{lemma7.2}. 

Now suppose $\dim_{k_\p }\Bigl(\Sel_{{S\ell_1\ell_2}}\bigl(K,\overline{T}_{f,\p}^\dagger\bigr)^\epsilon\Bigr)<\dim_{k_\p }\Bigl(\Sel_{{S\ell_1\ell_2}}\bigl(K,\overline{T}_{f,\p}^\dagger\bigr)^{-\epsilon}\Bigr)$. By Proposition \ref{prop masoero} and the non-triviality of $\kappa_{S\ell_1\ell_2}^\star$, we have $-\epsilon=\epsilon_{S\ell_1\ell_2}$. Then Proposition \ref{prop masoero} gives  
\begin{equation} \label{eq7.6}
\dim_{k_\p }\Bigl(\Sel_{{S\ell_1\ell_2},(\cB_{S\ell_1\ell_2}^\star)}\bigl(K,\overline{T}_{f,\p}^\dagger\bigr)^\epsilon\Bigr)\leq 
\nu_{S\ell_1\ell_2}^\star
\end{equation}
and 
\begin{equation} \label{eq7.7}
\dim_{k_\p }\Bigl(\Sel_{{S\ell_1\ell_2}}\bigl(K,\overline{T}_{f,\p}^\dagger\bigr)^{-\epsilon}\Bigr)=\nu_{S\ell_1\ell_2}^\star+1.
\end{equation}
Now $\dim_{k_\p}\Bigl(\Sel_{{S\ell_1}}\bigl(K,\overline{T}_{f,\p}^\dagger\bigr)^{-\epsilon}\Bigr)\geq\nu_{S\ell_1\ell_2}^\star+1$ by part (1) of Lemma \ref{lemma7.2} and \eqref{eq7.7}, so from \eqref{eq7.3} we get $\dim_{k_\p}\Bigl(\Sel_{{S}}\bigl(K,\overline{T}_{f,\p}^\dagger\bigr)^{-\epsilon}\Bigr)\geq\nu_{S\ell_1\ell_2}^\star+1$. Thus, by \eqref{eq7.1} we have 
\[ \dim_{k_\p }\Bigl(\Sel_{{S}}\bigl(K,\overline{T}_{f,\p}^\dagger\bigr)^\epsilon\Bigr)\geq\nu_{S\ell_1\ell_2}^\star+2. \] 
On the other hand, equation \eqref{eq7.2} then implies $\dim_{k_\p}\Bigl(\Sel_{{S\ell_1}}\bigl(K,\overline{T}_{f,\p}^\dagger\bigr)^\epsilon\Bigr)\geq\nu_{S\ell_1\ell_2}^\star+1$. Finally, combining \eqref{eq7.6} and \eqref{inclusion} yields
$\dim_{k_\p}\Bigl(\Sel_{{S\ell_1}}\bigl(K,\overline{T}_{f,\p}^\dagger\bigr)^\epsilon\Bigr)\leq\nu_{S\ell_1\ell_2}^\star$, which contradicts the previous inequality. \end{proof}

\begin{remark} \label{strategy-rem}
Theorem \ref{kol-thm}, whose proof requires $p>k+1$ (\emph{cf.} condition (H-$p$)), should be regarded as complementary to \cite[Theorem 3.35]{LV-TNC}, where Kolyvagin's conjecture was proved for $p<k$ under the congruence condition $k\equiv2\pmod{2(p-1)}$. In order to elaborate on this point, recall that this technical assumption, which essentially amounts to selecting a component of the $p$-adic weight space of Coleman--Mazur (\cite{ColMaz}), is needed in \cite{LV-TNC} because of the $p$-adic deformation approach adopted there, which is completely different from the strategy we follow in this paper. More precisely, the arguments in \cite{LV-TNC} exploit the arithmetic of $p$-adic Hida families (essentially in the guise of specialisation results for big Heegner points due to Howard, Castella and Ota, \emph{cf.} \cite{CasHeeg}, \cite{Howard-derivatives}, \cite{Ota-JNT}) and uses the results in weight $2$ from \cite{SZ} and \cite{zhang-selmer} as a ``bridge'' to the higher weight case.
\end{remark}




\section{On the Tamagawa number conjecture for modular motives} \label{tamagawa-sec}

The Tamagawa number conjecture of Bloch and Kato (\cite{BK}) predicts formulas for special values of $L$-functions of motives and can be viewed as a vast generalisation of the analytic class number formula and of the Birch--Swinnerton-Dyer conjecture for abelian varieties. The conjecture of Bloch--Kato, which was originally expressed in terms of Haar measures and Tamagawa numbers, was later reformulated and extended by Fontaine and Perrin-Riou (\cite{FPR}; \emph{cf.} also \cite{fontaine-bourbaki}) using the language of determinants of complexes and Galois cohomology; similar ideas were developed also by Kato (\cite{Kato-kodai}, \cite{Kato-Iwasawa}). 

In this section, we describe an application of the main result of this paper to the $p$-part of the Tamagawa number conjecture ($p$-TNC) for modular motives in analytic rank $1$. In doing this, we follow \cite{LV-TNC}, to which we systematically refer for definitions, details and proofs.

\subsection{Statement of $p$-TNC for modular motives} \label{TNC-statement-subsec}

Let $f\in S_k(\Gamma_0(N))$ be a newform of weight $k\geq4$ and square-free level $N$. We freely use the notation introduced in previous sections; in particular, $F$ is the Hecke field of $f$, whose ring of integers will be denoted by $\cO_F$, and $\p$ is a prime of $F$ above the prime number $p$. Set $F_\infty\defeq F\otimes_\Q\R$, $F_p\defeq F\otimes_\Q\Q_p$ and $\cO_p\defeq\cO_F\otimes_\Z\Z_p$. We attach to $f$ and $p$ the following objects.

\begin{itemize}
\item The motive $\MM=(X,\Pi,k/2)$ of $f$. This is a Grothendieck (\emph{i.e.}, homological) motive defined over $\Q$ with coefficients in $F$, equipped with its \'etale realization $V_p$ for each prime number $p$ (which is an $F_p$-module), its Betti (\emph{i.e.}, singular) realization and its de Rham realization (which are $F$-vector spaces), and comparison isomorphisms between these realisations. Here $X$ is the Kuga--Sato variety of level $N$ and weight $k$, while $\Pi$ is a projector on the ring of correspondences of $X$; see \cite[\S 2.2 and \S 2.4]{LV-TNC} for details. 
\item The (Bloch--Kato) Shafarevich--Tate group $\Sha_p^{\BKK}(\Q,\MM)$ of $\MM$ at $p$, which is defined as the quotient of the Bloch--Kato Selmer group of $V_p/T_p$ by its maximal $p$-divisible subgroup, where $T_p$ is a suitable Galois-stable $\cO_p$-lattice in $V_p$ (see \cite[\S 2.18]{LV-TNC}). The group $\Sha_p^{\BKK}(\Q,\MM)$ is finite. 
\item For every place $v$ of $\Q$ and prime $p$, a Tamagawa $\cO_p$-ideal $\Tam_v^{(p)}(\MM)$, whose definition is recalled in \cite[\S 2.21]{LV-TNC} (in particular, $\Tam_v^{(p)}(\MM)=\cO_p$ for all but finitely many $v$). 
\item The $p$-torsion part $\Tors_p(\MM)$ of $\MM$ (see \cite[\S 2.23.1]{LV-TNC}).
\item The period $\Omega_\MM\in(F\otimes_\Q\C)^\times$ coming from the comparison isomorphism between Betti and de Rham realizations (see \cite[\S 2.23.3]{LV-TNC}).
\item The motivic cohomology group $H^1_\mot(\Q,\MM)$, defined in \cite[\S 2.6]{LV-TNC}. It is a conjecturally finite-dimensional $F$-vector space (see, \emph{e.g.}, \cite[Conjecture 2.11]{LV-TNC}); assuming this finite-dimensionality, we set 
\[ r_\alg(\MM)\defeq\dim_F\bigl(H^1_\mot(\Q,\MM)\bigr). \]
The $F_\infty$-module $H^1_\mot(\Q,\MM)\otimes_FF_\infty$ is equipped with a conjecturally non-degenerate height pairing \emph{\`a la} Gillet--Soul\'e (see, \emph{e.g.}, \cite[\S 2.7]{LV-TNC}); we write $\Reg_{\mathscr B}(\MM)$ for the determinant of this pairing with respect to an $F$-basis $\mathscr B$ of $H^1_\mot(\Q,\MM)$, so that $\Reg_{\mathscr B}(\MM)\not=0$ if the pairing is non-degenerate. The $F_p$-module $H^1_\mot(\Q,\MM)\otimes_FF_p$ is endowed with a $p$-adic regulator map 
\[ \reg_p:H^1_\mot(\Q,\MM)\otimes_FF_p\longrightarrow H^1_f(\Q,V_p) \]
with values in the Bloch--Kato Selmer group $H^1_f(\Q,V_p)$ of $V_p$; this map is conjectured to be an isomorphism of $F_p$-modules (see \cite[Conjecture 2.47]{LV-TNC}).  
\item The \emph{completed} $L$-function $\Lambda(\MM,s)$ of $\MM$, which is an entire function on $\C$. We write $r_\an(\MM)$ (respectively, $\Lambda^*(\MM,0)$) for the order of vanishing (respectively, the leading term of the Taylor expansion) of $\Lambda(\MM,s)$ at $s=0$ (see \cite[\S 2.9]{LV-TNC}). 
\end{itemize}
The formulation, in the vein of Fontaine--Perrin-Riou, of the Tamagawa number conjecture for the motive $\MM$ can be found in \cite[Conjecture 2.52]{LV-TNC}. Here we recall an equivalent, more explicit version of its $p$-part that holds under certain arithmetic assumptions and involves all the ingredients in the previous list. In the statement below, for a finite $\cO_p$-module $M$ we denote by $\mathcal{I}(M)$ the $\cO_p$-ideal such that $\ord_\p\bigl(\mathcal{I}(M)\bigr)=\mathrm{length}_{\cO_\p}(M)$ for each prime $\p$ of $F$ above $p$, where $\cO_\p$ is the completion of $\cO_F$ at $\p$ and $\ord_\p$ is the $\p$-adic valuation.

\begin{theorem} \label{motivesthm}
Assume that 
\begin{enumerate}
\item the Gillet--Soul\'e height pairing on $H^1_\mot(\Q,\MM)\otimes_FF_\infty$ is non-degenerate (i.e., \cite[Conjecture 2.17]{LV-TNC} holds true);
\item the rationality conjecture of Beilinson and Deligne (\cite[Conjecture 2.39]{LV-TNC}) holds true for $\MM$;
\item the $p$-adic regulator $\reg_p$ is an isomorphism (i.e., the $p$-part of \cite[Conjecture 2.47]{LV-TNC} holds true for $K=\Q$).
\end{enumerate}
The $p$-part of the Tamagawa number conjecture for $\MM$ is equivalent to the equality 
\begin{equation} \label{eqBK}
\biggl(\frac{\Lambda^*(\MM,0)}{\Omega_\MM\cdot\Reg_{\mathscr B}(\MM)}\biggr)=\frac{\mathcal{I}\bigl(\Sha_p^{\BKK}(\Q,\MM)\bigr)\cdot\mathcal{I}_p(\gamma_f)\cdot\prod_{v\in S}\mathrm{Tam}_v^{(p)}(\MM)}{\bigl(\det(\mathtt{A}_{\tilde{\mathscr{B}}})\bigr)^2\cdot\Tors_p(\MM)} \tag{$p$-$\mathrm{TNC}_{\mathscr B}$}
\end{equation}
of fractional $\cO_p$-ideals.  
\end{theorem}
The terms $\mathcal{I}_p(\gamma_f)$ and $\mathtt{A}_{\tilde{\mathscr B}}\in\GL_{r_\alg(\MM)}(F_p)$, whose description is given in \cite[\S 2.23]{LV-TNC}, should be viewed as ``correction factors'' that take care of the fact that the definitions of some of the objects appearing in Theorem A involve choices (not reflected in the notation) of suitable bases; with these two extra ingredients in the picture, it can be checked that the validity of the resulting formula is independent of such choices.
\begin{proof}[Proof of Theorem \ref{motivesthm}] This is \cite[Theorem 2.79]{LV-TNC}. \end{proof}

\subsection{On the $p$-TNC for $\MM$ in analytic rank $1$}

We state our result on the $p$-part of the Tamagawa number conjecture for $\MM$ when the analytic rank of $\MM$ is $1$.

\subsubsection{$p$-TNC for $\MM$ in analytic rank $1$: assumptions} \label{TNC-ass}

In addition to Assumption \ref{ass}, which is in force until the end of this article, we need two general hypotheses in arithmetic algebraic geometry:
\begin{itemize}
\item an integral variant of the $p$-adic regulator map $\reg_p$ is an isomorphism;
\item $p$-adic Abel--Jacobi maps are injective on suitable modules of Heegner-type cycles.
\end{itemize}
These are specific instances of major conjectures that are expected to hold true for all primes $p$: as in \cite{LV-TNC}, we assume their validity whenever necessary.

As in the rest of the paper, $\p$ is a prime of $F$ above $p$. Write $D_F$ for the discriminant of $F$ and $c_f$ for the index of the order $\Z\bigl[a_n(f)\mid n\geq1\bigr]$ in $\cO_F$. As in \cite{LV-TNC}, we work under the following additional assumptions on the pair $(f,p)$:
\begin{itemize}
\item $N\geq3$ is square-free;
\item $p\nmid 6ND_F\mathrm{N}(\mathfrak a_{f,\Gamma(N)})c_f$;
\item $f$ is $p$-isolated, \emph{i.e.}, there are no non-trivial congruences modulo $p$ between $f$ and normalized eigenforms in $S_k(\Gamma_0(N))$;
\item $\bar\rho_\p$ is ramified at all primes dividing $N$ for all primes $\p$ of $F$ above $p$;
\end{itemize}
Here $\mathrm{N}(\mathfrak a_{f,\Gamma(N)})$ is the norm of a certain ideal $\mathfrak a_{f,\Gamma(N)}$ of $\cO_F$ defined in \cite[\S 4.3.6]{LV-TNC}. Thanks to a result of Ribet (see, \emph{e.g.}, \cite[Theorem 3.8]{LV-TNC}), the second and third conditions rule out only finitely many ``exceptional'' prime numbers $p\geq k$, which can be explicitly characterized (\emph{cf.} \cite[\S 3.2.1]{LV-TNC}). On the other hand, the residual ramification properties at the primes dividing $N$ are a higher weight counterpart of an analogous condition appearing in W. Zhang's paper on Kolyvagin's conjecture and the $p$-part of the Birch--Swinnerton-Dyer formula for rational elliptic curves (\cite{zhang-selmer}).

\subsubsection{$p$-TNC for $\MM$ in analytic rank $1$: statement} 

Write $\Sha_p^{\Nek}(\Q,\MM)$ for the Shafarevich--Tate group of Nekov\'a\v{r}, which is the quotient of the Bloch--Kato Selmer group of $V_p/T_p$ by the image of a certain $p$-adic Abel--Jacobi map (see, \emph{e.g.}, \cite[\S 4.5.1]{LV-TNC}). There is a natural surjection $\Sha_p^{\Nek}(\Q,\MM)\twoheadrightarrow\Sha_p^{\BKK}(\Q,\MM)$ of $\cO_p$-modules. Recall that $r_\alg(\MM)$ (respectively, $r_\an(\MM)$) denotes the algebraic (respectively, analytic) rank of $\MM$. 

\begin{theorem}[$p$-TNC for $\MM$] \label{p-TNC}
Suppose that $r_\an(\MM)=1$. Under the assumptions in \S \ref{TNC-ass}, the following results hold:
\begin{enumerate}
\item $r_\alg(\MM)=1$; 
\item $\Sha_p^{\BKK}(\Q,\MM)=\Sha_p^{\Nek}(\Q,\MM)$;  
\item the $p$-part of the Tamagawa number conjecture for $\MM$ is true.
\end{enumerate}
\end{theorem}

\begin{proof} This is \cite[Theorem 4.41]{LV-TNC}. Notice that, in particular, part (3) is proved by checking that equality \eqref{eqBK} in Theorem \ref{motivesthm} holds true in this case. \end{proof}

This is Theorem B in the introduction. It is worth observing that, albeit not available in the literature and never formulated for the motive $\MM$, parts (1) and (2) were essentially already known, thanks to a combination of work of Nekov\'a\v{r} on the arithmetic of Chow groups of Kuga--Sato varieties (\cite{Nek}) and analytic results by Bump--Friedberg--Hoffstein (\cite{BFH}), Murty--Murty (\cite{MM-derivatives}) and Waldspurger (\cite{Waldspurger}); thus, as for \cite[Theorem 4.41]{LV-TNC}, the novelty of Theorem \ref{p-TNC} lies almost entirely in part (3). 

\subsubsection{$p$-TNC for $\MM$ in analytic rank $1$: remarks on the proof} \label{TNC-remarks}

Theorem \ref{p-TNC} can be proved exactly as \cite[Theorem 4.41]{LV-TNC}, observing that \cite[Theorem 3.35]{LV-TNC}, which proves Kolyvagin's conjecture in higher weight under the congruence condition $k\equiv2\pmod{2(p-1)}$, must be replaced here by Theorem \ref{kol-thm}. More precisely, the above-mentioned technical condition was needed in \cite{LV-TNC} to apply, in particular, results by Skinner--Urban on the (cyclotomic) Iwasawa main conjecture for modular forms (\cite{SU}). Actually, a closer inspection reveals that this congruence assumption is not really employed in the proof of the results from \cite{SU} we are interested in, so that we do not need to impose it in our present paper (notice, however, that it cannot be removed from \cite{LV-TNC}, as it is crucially used in the Hida-theoretic arguments alluded to in Remark \ref{strategy-rem}). Finally, we would like to point out that results from \cite{SU} are used in \cite{wang} as well without any congruence condition whatsoever on $k$ and $p$.

\section{Other consequences on the arithmetic of $f$} \label{consequences-sec}

In this final section, we collect some further consequences of our result on Kolyvagin's conjecture. Since proofs can be found in \cite{LV-TNC}, here we content ourselves with describing the statements of these by-products of Theorem \ref{kol-thm}. We remark that Assumption \ref{ass} is in force.

\subsection{Structure of Bloch--Kato--Selmer groups}

Our usual notation being in force, let $\p$ be a prime of $F$ above $p$. As before, let $H^1_f(K,A_\p)^\pm$ denote the $\pm1$-eigenspaces of complex conjugation acting on $H^1_f(K,A_\p)$; set
\[ r_\p(f/K)\defeq \corank_{\cO_\p}H^1_f(K,A_\p)\in\N,\quad r^\pm_\p(f/K)\defeq \corank_{\cO_\p}H^1_f(K,A_\p)^\pm\in\N \]
and observe that
\[ r_\p(f/K)=r^+_\p(f/K)+r^-_\p(f/K). \]
Recall the vanishing order $\nu_S$ of $\kappa_S$ from Definition \ref{vanishing-def} and, in line with the notation in \cite[\S 5.1]{LV-TNC}, set $\nu_\infty\defeq\nu_1$. It is convenient to introduce the sign
\begin{equation} \label{varepsilon-infty-eq}
\epsilon_\infty\defeq\mathrm{sign}\bigl(\epsilon(f)\cdot(-1)^{\nu_\infty+1}\bigr)\in\{\pm\}, 
\end{equation}
which will appear in the next result. 

The prime $p$, which is unramified in $F$, is a uniformizer for $\cO_\p$. Set $\epsilon\defeq\mathrm{sign}\bigl(\epsilon(f)\bigr)\in\{\pm\}$. Write 
\[ H^1_f(K,A_\p)^\pm\simeq (F_\p/\mathcal{O}_\p)^{r_\p^\pm(f/K)}\oplus\mathcal X_\p^\pm \] 
where $\mathcal X_\p^\pm$ is a finite $\cO_\p$-module, then consider splittings
\begin{align}
\mathcal X_\p^{-\epsilon}&\simeq(\cO_\p/p^{n_1}\cO_\p)^2\oplus(\cO_\p/p^{n_3}\cO_\p)^2\oplus\dots\\
\intertext{and}
\mathcal X_\p^{\epsilon}&\simeq(\cO_\p/p^{n_2}\cO_\p)^2\oplus(\cO_\p/p^{n_4}\cO_\p)^2\oplus\dots
\end{align}
of $\cO_\p$-modules. Finally, define the integers $N_i\in\N$ as in \cite[\S 4.7]{LV-TNC} (\emph{cf.} \cite[Theorem 4.31]{LV-TNC}) and let $\epsilon_\infty\in\{\pm\}$ be the sign from \eqref{varepsilon-infty-eq}. 

\begin{theorem} \label{main-vanishing-thm} 
\begin{enumerate}
\item $r_\p^{\epsilon_\infty}(f/K)=\nu_\infty+1$ and $r_\p^{-\epsilon_\infty}(f/K)\leq\nu_\infty$.
\item $\nu_\infty=\max\bigl\{r^+_\p(f/K),r^-_\p(f/K)\bigr\}-1$.
\item $n_i=N_i$ for all $i>\nu_\infty+1$. 
\item $0\leq\nu_\infty-r_\p^{-\epsilon_\infty}(f/K)\equiv0\pmod{2}$.
\end{enumerate}
\end{theorem}

\begin{proof} This is \cite[Theorem 5.4]{LV-TNC}. \end{proof}

\subsection{$\p$-parity results}

Fix a prime $\p$ of $F$ above $p$ and set
\begin{equation} \label{corank-eq}
r_\p(f)\defeq\corank_{\cO_\p}H^1_f(\Q,A_\p)\in\N. 
\end{equation}
As usual, let $\epsilon(f)\in\{\pm1\}$ be the root number of $f$.

\begin{theorem} \label{parity-thm} 
$(-1)^{r_\p(f)}=\epsilon(f)$. 
\end{theorem}

\begin{proof} See the proof of \cite[Theorem 6.1]{LV-TNC}. \end{proof}

This is a higher weight counterpart of results of Nekov\'a\v{r} on the $p$-parity conjecture for elliptic curves and modular abelian varieties (\cite{Nekovar-parity}, \cite{Nekovar-parity2}); the reader is referred to, \emph{e.g.}, \cite[Ch. 12]{Nek-Selmer}, \cite{Nekovar-growth}, \cite{NP} for results in the same vein for  (Hilbert) modular forms.

\subsection{$\p$-converse theorems}

Assume that $r_\p(f)=1$, so that $\epsilon(f)=-1$ by Theorem \ref{parity-thm}. Consider the imaginary quadratic fields $K$ satisfying the following two conditions:
\begin{itemize}
\item the primes dividing $Np$ split in $K$;
\item $r_\an(f^K)=0$, where $f^K$ is the twist of $f$ by the Dirichlet character associated with $K$.
\end{itemize} 
Denote by $\mathscr I_1(f,p)$ the set of all such fields. It follows from \cite[p. 543, Theorem, (ii)]{BFH} (\emph{cf.} also \cite{Waldspurger}) that $\mathscr I_1(f,p)\not=\emptyset$.

As in \cite[\S 7.1.2]{LV-TNC}, we need hypotheses on the non-degeneracy of the Gillet--Soul\'e height pairings ${\langle\cdot,\cdot\rangle}_{\GS}$ introduced in \cite[\S 2.7]{LV-TNC} and alluded to in \S \ref{TNC-statement-subsec}:
\begin{itemize}
\item[(\texttt{GS})] there is $K\in\mathscr I_1(f,p)$ such that ${\langle\cdot,\cdot\rangle}_{\GS}$ is non-degenerate on $\Heeg_{K,N}\otimes_\Z\,\R$, where $\Heeg_{K,N}$ is the \emph{Heegner module of level $N$} defined in \cite[\S 4.1.3]{LV-TNC}.
\end{itemize}
As before, $\p$ is a prime of $F$ above $p$. Set
\[ X_{\p}(\Q)\defeq\Lambda_{\p}(\Q)\otimes_{\Z}\Q=\Lambda_{\p}(\Q)\otimes_{\cO_\p}\!F_\p \]
and let $r_\p(f)\in\N$ be the corank from \eqref{corank-eq}. The following result is an analogue in higher weight of the algebraic part of \cite[Theorem 1.4, (i)]{zhang-selmer}.

\begin{theorem} \label{main-selmer-thm} 
If $r_\p(f)=1$, then 
\begin{enumerate}
\item $\dim_{F_\p}\bigl(X_\p(\Q)\bigr)=1$;
\item $\Sha_\p^{\Nek}(\Q,\MM)$ is finite.
\end{enumerate}
\end{theorem}

\begin{proof} This is \cite[Theorem 7.3]{LV-TNC}. \end{proof}

We can also offer the following $\p$-converse theorem.

\begin{theorem} \label{main-converse-Q-thm}
If $r_\p(f)=1$ and $(\mathtt{GS})$ holds, then $r_\an(f)=1$.
\end{theorem}

\begin{proof} This is \cite[Theorem 7.4]{LV-TNC}. \end{proof}

\begin{remark} \label{K-rem}
If we knew that Gillet--Soul\'e height pairings are non-degenerate (at least on the $\R$-vector space $\Heeg_{K,N}\otimes_\Z\,\R$, or in full generality, as predicted by the conjectures in \cite{Beilinson}, \cite{bloch-height}, \cite{GS-2}), then Theorem \ref{main-converse-Q-thm} would become unconditional. Unfortunately, non-degeneracy results of this kind seem to lie well beyond the scope of current techniques.
\end{remark}

\bibliographystyle{amsplain}
\bibliography{Perrin-Riou}

\end{document}